\pdfoutput=1
\documentclass[reqno,a4paper,11pt]{amsart}
\usepackage{amsfonts}
\usepackage{amsmath}
\usepackage{amssymb}
\usepackage[pdftitle={An analytic invariant of G2-manifolds},hidelinks]{hyperref}
\usepackage[utf8]{inputenc}

\newcommand{\showdate}{false}

\usepackage{amssymb}
\usepackage[utf8]{inputenc}
\usepackage{a4wide}
\usepackage{enumitem}
\usepackage{tikz}
\usetikzlibrary{matrix}
\usepackage{color}
\usepackage{multido}
\usepackage{ifthen}
\RequirePackage{xspace}

\title{Eta-invariants of extra-twisted connected sums}

\ifthenelse{\boolean{\showdate}}{\date{\today}}{}

\newcommand{\ignore}[1]{}

\DeclareMathOperator{\re}{Re}
\DeclareMathOperator{\ind}{ind}
\DeclareMathOperator{\tr}{tr}
\DeclareMathOperator{\im}{Im}

\DeclareMathOperator{\vol}{vol}

\DeclareMathOperator{\rk}{rk}

\DeclareMathOperator{\supp}{supp}
\DeclareMathOperator{\SF}{sf}
\DeclareMathOperator{\Spec}{spec}
\renewcommand\Re{\operatorname{Re}}
\renewcommand\Im{\operatorname{Im}}
\newcommand\Spin{\mathrm{Spin}}
\renewcommand\div{\operatorname{div}}
\newcommand{\ie}{\emph{i.e.} }
\newcommand{\eg}{\emph{e.g.} }
\newcommand\cf{see }
\newcommand{\sign}{\sigma}
\newcommand{\kd}{\Sigma}

\newcommand{\lagr}{L}
\newcommand{\latc}{\textup{L}}

\newcommand{\ant}{-\id_{S^1}}
\newcommand{\hk}{hyper-Kähler\xspace}
\DeclareMathAlphabet{\matheur}{U}{eur}{m}{n}
\newcommand{\hkr}{\matheur{r}}

\DeclareMathOperator{\Pic}{Pic}
\newcommand{\PP}{\mathbb{P}}

\newcommand{\CP}{\mathbb{CP}}
\newcommand{\APS}{\mathrm{APS}}

\newcommand{\sm}[1]{\left(\begin{smallmatrix} #1 \end{smallmatrix} \right)}

\newcommand{\ahat}{\widehat{A}}

\newcommand{\mmod}{\!\!\mod}

\newcommand{\C}{\mathbb{C}}
\newcommand{\Z}{\mathbb{Z}}

\newcommand{\bbr}{\mathbb{R}}
\newcommand{\R}{\mathbb{R}}
\newcommand{\bbc}{\mathbb{C}}

\newcommand{\bbo}{\mathbb{O}}

\newcommand{\trivr}{\underline{\bbr}}

\newcommand{\del}{\partial}

\newcommand{\gstr}{$G$\nobreakdash-\hspace{0pt}structure}
\newcommand{\gtstr}{$G_{2}$\nobreakdash-\hspace{0pt}structure}

\newcommand{\gtmfd}{$G_{2}$\nobreakdash-\hspace{0pt}manifold}

\newcommand{\norm}[1]{\left\Vert #1\right \Vert}

\newcommand{\anglen}{u}
\newcommand{\anglex}{v}
\newcommand{\lnx}{\xi}
\newcommand{\lnn}{\zeta}

\newcommand{\dirac}{\slashed{\partial}}

\newcommand\abs[1]{\left|#1\right|}
\newcommand\ev{{\mathrm{ev}}}
\newcommand\odd{{\mathrm{odd}}}
\newcommand\id{\mathrm{id}}
\newcommand\Aut{\operatorname{Aut}}

\newcommand\Span{\operatorname{span}}
\newcommand\Sign{\operatorname{sign}}
\newcommand\punkt{\mathord{\,\cdot\,}}
\newcommand\Spinor{S}

\let\eps\varepsilon
\let\phy\varphi
\let\thet\vartheta
\let\<\langle
\let\>\rangle
 
\newcommand\K{\Sigma} %
\newcommand\X{X} %
\renewcommand\dirac{D}
\newcommand\D{\mathbb D}
\newcommand\trang{\omega} %

\newtheorem{Thm}{Theorem}
\newtheorem{Cor}[Thm]{Corollary}
\newtheorem{Qstn}[Thm]{Question}
\newtheorem{thm}{Theorem}[section]
\newtheorem{prop}[thm]{Proposition}
\newtheorem{lem}[thm]{Lemma}

\theoremstyle{definition}
\newtheorem{dfn}[thm]{Definition}

\newtheorem{matchingproblem}[thm]{Matching problem}
\theoremstyle{remark}
\newtheorem{rmk}[thm]{Remark}
\newtheorem{ex}[thm]{Example}

\setlist{leftmargin=*}
\newcommand{\thefigures}{
\protect \begin{figure}
\protect \begin{minipage}{0.48\textwidth}
\centering
\begin{tikzpicture}[x=0.8cm,y=0.8cm]
  \multido{\ix=-2+2}{3}{
    \multido{\iy=-2+2}{3}{
      \fill (\ix,\iy) circle (2pt) ;
    }
  }
  \multido{\ix=-1+2}{2}{
    \multido{\iy=-1+2}{2}{
      \fill (\ix,\iy) circle (2pt) ;
    }
  }
  \draw (0,0) -- (0,2) ;
  \draw (0,0) -- (1,-1) ;
  \draw (0,0) -- (1,1) ;
  \draw (0,0) -- (2,0) ;
  \begin{scope}[->, line width=1pt]
    \draw[color=blue] (0,0) -- (90:1.131cm)
	node[above left, color=black] {$\partial_{u_+}$} ;
    \draw[color=blue] (0,0) -- (-45:1.131cm)
	node[below right, color=black] {$\partial_{u_-}$} ;
    \draw[color=red] (0,0) -- (45:1.131cm)
	node[above right, color=black] {$\partial_{v_-}$} ;
    \draw[color=red] (0,0) -- (0:1.131cm)
	node[below right, color=black] {$\partial_{v_+}$} ;
  \end{scope}
  \fill (0,0) circle (3pt) ;
  \draw[->] (0,0) ++(0:0.6cm) arc (0:45:0.6cm) ;
  \node at (22.5:0.85cm) {$\vartheta$} ;
\end{tikzpicture}
\caption{\texorpdfstring{$\vartheta = \displaystyle \frac{\pi}{4}$}{angle pi/4}}
\label{fig:1/4}
\protect \end{minipage}\hfill
\protect \begin{minipage}{0.48\textwidth}
\centering
\begin{tikzpicture}[x=1cm,y=0.577cm]
  \multido{\ix=-2+2}{3}{
    \multido{\iy=-2+2}{3}{
      \fill (\ix,\iy) circle (2pt) ;
    }
  }
  \multido{\ix=-1+2}{2}{
    \multido{\iy=-3+2}{4}{
      \fill (\ix,\iy) circle (2pt) ;
    }
  }
  \draw (0,0) -- (0,2) ;
  \draw (0,0) -- (1,-3) ;
  \draw (0,0) -- (1,1) ;
  \draw (0,0) -- (2,0) ;
  \begin{scope}[->, line width=1pt]
    \draw[color=blue] (0,0) -- (90:1.154cm)
	node[above, color=black] {$\partial_{u_+}$} ;
    \draw[color=blue] (0,0) -- (-60:1.154cm)
	node[right, color=black] {$\partial_{u_-}$} ;
    \draw[color=red] (0,0) -- (30:1.154cm)
	node[above right, color=black] {$\partial_{v_-}$} ;
    \draw[color=red] (0,0) -- (0:1.154cm)
	node[below right, color=black] {$\partial_{v_+}$} ;
  \end{scope}
  \fill (0,0) circle (3pt) ;
  \draw[->] (0,0) ++(0:0.75cm) arc (0:30:0.75cm) ;
  \node at (15:1cm) {$\vartheta$} ;
\end{tikzpicture}
\caption{\texorpdfstring{$\vartheta = \displaystyle \frac{\pi}{6}$}{angle pi/6}}
\label{fig:1/6}
\protect \end{minipage}
\protect \end{figure}
}

\definecolor{darkgreen}{rgb}{0,0.5,0}
\definecolor{darkred}{rgb}{0.7,0,0}
\definecolor{darkblue}{rgb}{0,.2,.7}

\begin{document}

\title{An analytic invariant of $G_2$ manifolds}

\author{Diarmuid Crowley, Sebastian Goette and Johannes Nordstr\"om}


\address{\hspace{-12pt}School of Mathematics and Statistics\\
University of Melbourne\\Parkville VIC 3010\\Australia}
\email{dcrowley@unimelb.edu.au}

\address{\hspace{-12pt}Mathematisches Institut\\ Universit\"at Freiburg\\ Ernst-Zermelo-Str.~1, 79104 Freiburg, Germany} \email{sebastian.goette@math.uni-freiburg.de}

\address{\hspace{-12pt}Department of Mathematical Sciences\\
University of Bath\\
Bath BA2 7AY\\
UK}\email{j.nordstrom@bath.ac.uk}

\begin{abstract}
  We prove that the moduli space of holonomy $G_2$-metrics
  on a closed $7$-manifold
  can be disconnected by presenting a number of explicit examples.

  We detect different connected components of the $G_2$-moduli space
  by defining an analytic refinement~$\bar \nu(M, g) \in \Z$
  of the defect invariant~$\nu(M,\phy)\in\Z/48$
  of $G_2$-structures~$\phy$ on a closed 7-manifold~$M$
  introduced by the first and third authors.
  The $\bar \nu$-invariant is defined using $\eta$-invariants and Mathai-Quillen
  currents on~$M$ and we compute it for twisted connected sums \`a la Kovalev,
  Corti-Haskins-Nordstr\"om-Pacini and extra-twisted
  connected sums as constructed by 
  the second and third authors.
  In particular, we find examples of $G_2$-holonomy metrics in different
  components of the moduli space where the associated $G_2$-structures are
  homotopic and other examples
  where they are not. %
\end{abstract}

\maketitle

Recent years have seen progress in the construction and description
of closed $G_2$-manifolds. Apart from Joyce's Kummer
construction~\cite{G2I, Joyce}, one also has the twisted connected sum
construction inspired by Donaldson, implemented by Kovalev~\cite{Kovalev}
and generalised by Corti, Haskins, Nordstr\"om, Pacini~\cite{CHNP}.
Each $G_2$-holonomy metric corresponds to a \gtstr{} that is torsion-free.
With a large supply of $G_2$-holonomy metrics,
one can now ask if some different constructions
\begin{enumerate}
\item\label{Q1} lead to the same closed $7$-manifold up to diffeomorphism?
\item\label{Q2} if so, whether the underlying homotopy classes of
$G_2$-structures are equal (up to spin diffeomorphism)?
\item\label{Q3} if so, whether the two metrics lie in the same connected
component of the moduli space of $G_2$-holonomy metrics over the given class
  of $G_2$-structures?
\end{enumerate}
Note that by homotopy of \gtstr s we always mean the purely topological
condition of being able to connect two \gstr s by a path of \gtstr s.
On the other hand, the condition of two $G_2$ metrics being in the same
component of the moduli space corresponds to being able to join their
torsion-free \gtstr s (up to diffeomorphism) by a path of \gtstr s that are
themselves torsion-free.

For the twisted connected sum construction, Question\:\ref{Q1} is answered
affirmatively by
examples exhibited in \cite[Table 3]{CHNP} and \cite[Table 4]{CrN3},
making use of classification results for 2-connected 7-manifolds
of Wilkens \cite{wilk1, wilk2} (\cf Theorem~\ref{thm:torfreeclass}).
Regarding Question\:\ref{Q2},
two of the authors of this article defined
the $\Z/48$-valued $\nu$-invariant of \gtstr s on closed 7-manifolds~\cite{CrN}.
If the topology of the underlying manifold is sufficiently simple,
then~$\nu$ detects all $G_2$-structures up to homotopy and spin
diffeomorphism, see Proposition \ref{prop:nu_classifies}.
However, so far there have been no explicit examples of spin $7$-manifolds
where two different classes of $G_2$-structures admit $G_2$-holonomy metrics.

This paper is concerned with Questions\:\ref{Q2} and\:\ref{Q3}.
To address them, we introduce an integer-valued refinement~$\bar\nu$ of the
$\nu$-invariant, see Definition~\ref{MainDef}. It is normalised so that for a
metric $g$ with holonomy exactly~$G_2$ on a closed 7-manifold
and $\varphi$ the associated torsion-free \gtstr,
\begin{equation}
\label{eq:relation}
\nu(\varphi) = \bar\nu(g) + 24 \mod 48 .
\end{equation}
The invariant~$\bar\nu(g)$ is defined analytically using $\eta$-invariants
and can be computed from the Riemannian metric~$g$ alone, provided~$g$
has holonomy~$G_2$.
It is locally constant on the moduli space of $G_2$-metrics because
the $\eta$-invariant of the spin Dirac operator depends continuously on
$G_2$-metrics;
in this regard, our invariant resembles Kreck and Stolz's refined
Eells-Kuiper invariant~\cite[Definition~2.12]{KS3}.
On the other hand,
the linear combination of $\eta$-invariants defining $\bar\nu(g)$ does not
define a continuous function on the space of all metrics, nor is it even
integer-valued unless one adds an additional term that depends on a
$G_2$-structure and not just a metric, see Definition~\ref{dfn:nubar}.

A modified version of the twisted connected sum construction is
outlined in Section~\ref{Kap2}.
It involves gluing together a pair of manifolds of the form
$(V_+ \times S^1)/\Gamma_\pm$, where $V_\pm$ is an asymptotically cylindrical
Calabi-Yau manifold on which a cyclic group $\Gamma_\pm \cong \Z/k_\pm$ acts
by automorphisms. A key parameter of such an ``extra-twisted connected sum''
is the constant $\thet\in(0,\pi)$ corresponding to the angle between
the $S^1$ factors under the gluing, see Subsection \ref{subsec:glue}.
The case $k_+ = k_- = 1$ recovers the ordinary twisted
connected sums of \cite{Kovalev} and \cite{CHNP}; in this case $\thet$ is
forced to be $\frac{\pi}{2}$.
The cases when $k_+, k_- \leq 2$ are studied in detail in~\cite{xtcs}.
Note that we do not allow $\thet\in\pi\Z$; that would result in
``untwisted connected sums'' with infinite fundamental group,
which cannot have full holonomy~$G_2$.

Our main result can be stated as follows,
using notation introduced in Sections~\ref{Kap2} and~\ref{Kap3}.
The integer~$m_\rho(\latc;N_+,N_-)$ only depends on the positions
of the images~$N_\pm$ of~$H^2(V_\pm)$ in the K3 lattice~$\latc$
and on the gluing angle~$\thet$, see Definition~\ref{def2.6}.
In Definition \ref{Def:nupm},
we generalise the analytic invariant~$\bar\nu$
to manifolds with boundary, using APS boundary conditions
modified by a natural choice of Lagrangian subspaces
as suggested in~\cite{KiLe}.

\begin{Thm}\label{GluingThm}
  Let~$(M,g)$ be an extra-twisted connected sum.
  Let~$\rho=\pi-2\thet$, then
  \begin{equation*}
    \bar\nu(M, g)
    =\bar\nu(M_+,g)+\bar\nu(M_-,g)-72\frac\rho\pi+3\,m_\rho(\latc;N_+,N_-)\;.
  \end{equation*}
\end{Thm}

The analytic description of $\bar\nu$ makes it (and hence~$\nu$)
explicitly computable for the extra-twisted connected sums of~\cite{xtcs},
where $k_+$ and $k_-$ are both $\leq 2$.
Due to spectral symmetry, the contributions~$\bar\nu(M_\pm,g)$ vanish in these
cases.

\begin{Cor}\label{MainThm}
  Let~$(M,g)$ be an extra-twisted connected sum with $k_\pm\le 2$.
  Then
  \begin{equation*}
    \bar\nu(M, g)
    =-72\frac\rho\pi+3\,m_\rho(\latc;N_+,N_-)\;.
  \end{equation*}
\end{Cor}

For~$\thet=\frac\pi2$, we have~$\rho=0$ and~$m_\rho(\latc;N_+,N_-)=0$.
This gives the following refinement of the claim from
\mbox{\cite[Theorem~1.7]{CrN}} that any rectangular twisted connected sum
has $\nu = 24$.

\begin{Cor}\label{MainCor}
  If~$(M,g)$ is a rectangular twisted connected sum,
  that is, if~$\thet=\frac\pi 2$, then
  \begin{equation*}
    \bar\nu(M,g)=0\;.
  \end{equation*}
\end{Cor}

By applying Corollary \ref{MainThm} to examples described in detail
in~\cite[\S 8]{xtcs},
we can prove two statements concerning Questions~\ref{Q2} and~\ref{Q3}
above.
In both theorems, the diffeomorphism type of the manifold $M$ is
completely characterised by the given invariants
(see Theorem \ref{thm:torfreeclass}).

\begin{Thm}\label{Thm4.1}
  There is a closed 2-connected 7-manifold $M$ with~$H^4(M;\Z)\cong\Z^{97}$
  and first Pontrjagin class ~$p_1(M)=4a$ for a primitive class $a \in H^4(M;\Z)$,
  admitting two $G_2$-holonomy metrics whose associated 
  $G_2$-structures are not related by homotopy and diffeomorphism.
\end{Thm}

In particular, the two $G_2$-holonomy metrics in Theorem \ref{Thm4.1} must
be in different components of the $G_2$ moduli space because of the topological
obstruction. In the next theorem we distinguish components of the moduli
space even when there is no topological obstruction.

\begin{Thm}\label{Thm4.2}
  There is a closed 2-connected 7-manifold $M$ with~$H^4(M;\Z)\cong\Z^{109}$
  and $p_1(M)=4a$ for a primitive class $a \in H^4(M;\Z)$,
  admitting a homotopy class of $G_2$-structures
  over which the moduli space of $G_2$-metrics up to diffeomorphism
  has more than one connected component.
\end{Thm}

While the homotopy classes of the \gtstr s in Theorem~\ref{Thm4.1} can be
distinguished using just the $\nu$-invariant
of~\cite{CrN}, the invariant~$\bar\nu$ is needed to distinguish the components
of the moduli space in Theorem~\ref{Thm4.2}; see Example \ref{ex:pi6rk2}.
However, the only way we know
to compute the $\nu$-invariant of Example \ref{ex:pi4rk2}---on which
Theorem~\ref{Thm4.1} relies---is to use Corollary \ref{MainThm} to
compute~$\bar\nu$ and apply the equation~\eqref{eq:relation}. 
(See Remark \ref{rmk:pi4mixed} for another example of the same type.)

If~$k_\pm\le2$,
then~$\rho\in\bigl\{0,\pm\frac\pi3,\pm\frac\pi2,\pm\frac{2\pi}3\bigr\}$,
and~$\bar\nu(M,g)$ (and hence also~$\nu$) is divisible by~$3$.
Moreover, because both~$\rho$ and the contribution by the angles~$\alpha^-_1$,
\dots, $\alpha^-_{19}$ is bounded,
the invariant~$\bar\nu(M,g)$ can only attain finitely many values for these
extra-twisted connected sums.

\begin{Qstn}
  What is the range of~$\bar\nu$ on arbitrary $G_2$-manifolds?
  Is it finite?
\end{Qstn}

To answer this question, it would be helpful to know the $\bar\nu$-invariant
of Joyce's examples. %
For work in this direction see
Fornasin \cite{fornasin} and Scaduto \cite{scaduto}. %

In the sequel paper~\cite{GN},
the second and the third author compute $\nu$-invariants of
more general extra-twisted connected sums.
In those cases, the relevant operators on~$M_\pm$ no longer exhibit
spectral symmetry, so we can have~$\bar\nu(M_\pm,g)\ne0$ in
Theorem~\ref{GluingThm}.
Moreover, we there find examples where~$3\nmid\bar\nu(M,g)$.

\medskip

This paper is organised as follows.
In Section~\ref{Kap1}, we define the invariant~$\bar\nu$ and give simple
examples.
In Section~\ref{Kap2}, we describe the extra-twisted connected sum construction.
Section~\ref{Kap4} contains the examples mentioned in the theorems above.
We give a cohomological description of the signature $\eta$-invariant
in Section~\ref{Kap6}
using the gluing formula of Kirk and Lesch~\cite{KiLe}.
In Section~\ref{Kap3}, we derive the gluing formula for the spinor
$\eta$-invariant from Bunke's gluing formula~\cite{BuGlu}
and compute~$\bar\nu$ for twisted connected sums.

Except in Section \ref{Kap4}, %
$H^\bullet(\punkt)$ will always refer to
cohomology with real coefficients, which we identify with de Rham cohomology,
and with the space of harmonic forms, if the underlying space
is a compact manifold.

\smallskip
\noindent
{\bf Acknowledgements.} The authors thank Uli Bunke, Alessio Corti,
Mark Haskins, Matthias Lesch, Arkadi Schelling and Thomas Walpuski
for valuable discussions, and the referees for constructive comments.
SG and JN would like to thank the Simons foundation for its support of their
research under the Simons Collaboration on ``Special Holonomy in Geometry,
Analysis and Physics'' (grants \#488617, Sebastian Goette, and \#488631, Johannes
Nordstr\"om).

\section{The extended \texorpdfstring{$\nu$-}{nu-}invariant}\label{Kap1}

We recall from \cite{CrN} the definition of the $\Z/48$-valued invariant
$\nu$ of a \gtstr{} on a closed 7-manifold $M$, involving
a spin zero-bordism~$W$ of~$M$.
Using the Atiyah-Patodi-Singer index theorem,
we give an intrinsic description of~$\nu$ in terms of $\eta$-invariants
and a Mathai-Quillen current on~$M$, %
which moreover allows us to define a diffeomorphism invariant~$\bar\nu$
taking values in~$\Z$.

\subsection{An intrinsic formula}\label{Abs1.a}

Let~$M$ be a closed spin 7-manifold
with tangent bundle $TM\to M$
and with a fixed real spinor bundle~$SM\to M$,
which is of rank~$8$.
A {\em $G_2$-structure\/} on~$M$
can be identified with a Riemannian metric together with a unit spinor
field~$s\in\Gamma(SM)$ (up to sign),
see~\cite[Section~2.2]{CrN}.
We will consider $G_2$-structures up to homotopy and spin diffeomorphism.
Note that homotopy classes of \gtstr s correspond to homotopy classes of
non-vanishing spinor fields.
Hence, we may write~$\nu(s)$ for~$\nu(\phy)$.

As a spin 7-manifold, $M$ can be represented as the spin boundary
of a compact spin 8-manifold~$W$.
Let~$\chi(W)$ and~$\sigma(W)$ denote the Euler characteristic
and the signature of~$W$.
We identify~$SM$ with~$S^+W|_M$ in a natural way and extend~$s$
to~$\bar s\in\Gamma(S^+W)$.
The spinor bundles~$SM$ and~$S^+W$ can be given natural orientations,
where we follow the convention of~\cite[\S2.3]{CrN}.
Assuming that~$\bar s$ is transverse to the zero section~$W\subset S^+W$,
let~$n(\bar s)$ denote the number of zeros of~$\bar s$,
counted with sign.

\begin{dfn}[Crowley--Nordstr\"om {\cite{CrN}}]
\label{def:nu}
  The $\nu$-invariant of~$(M,s)$ is defined as
  \begin{equation*}
    \nu(s)=\chi(W)-3\sigma(W)-2n(\bar s)\in\Z/48\;.
  \end{equation*}
\end{dfn}

In fact, in~\cite[Definition~1.2]{CrN}
the manifold~$W$ is assumed to carry a $\Spin(7)$-structure
that induces the given $G_2$-structure on~$M=\partial W$.
We find it more convenient to work without this restriction.
This introduces the additional term~$-2n(\bar s)$,
see~\cite[Section~3.2]{CrN}.

It is proved in~\cite{CrN} that~$\nu$ is a well-defined invariant of
$G_2$-structures. It is patently invariant under homotopies and
spin diffeo\-morphism.
Moreover, for certain topologically simple 7-manifolds, 
$\nu$ is a complete invariant of $G_2$-structures up to homotopy and
spin diffeo\-morphism, \cf Proposition~\ref{prop:nu_classifies}.

We fix a Riemannian metric~$g^{TM}$,
which induces a metric~$g^{SM}$ and a connection~$\nabla^{SM}$ on~$SM$.
We also fix a Riemannian metric~$g^{TW}$ on~$W$
such that a collar neighbourhood of~$M=\partial W$ is isometric
to a product~$(M,g^{TM})\times[0,\eps)$.
Then we have an induced metric~$g^{S^+W}$ and a connection~$\nabla^{S^+W}$
on~$S^+W$ that restrict to~$g^{SM}$ and~$\nabla^{SM}$ over~$M=\partial W$.
Let~$e(\nabla^{S^+W})$ denote the Chern-Weil Euler form of~$S^+W$.

If~$Y\subset X$ is an oriented submanifold of an oriented manifold~$X$,
we regard the Dirac $\delta$-distribution along~$Y$ as a locally
integrable current~$\delta_Y$ of degree~$\dim X-\dim Y$ such that
\begin{equation*}
  \int_X\delta_Y\wedge\alpha=\int_Y\alpha|_Y
\end{equation*}
for all compactly supported~$\alpha\in\Omega^\bullet(X)$.
Let~$\psi(\nabla^{S^+W},g^{S^+W})$ denote
the Mathai-Quillen current on the total space of~$\pi\colon S^+W\to W$.
It is a locally integrable current of degree~$7$ on~$S^W$,
defined in~\cite[Section~7]{MQ}
and explained further in~\cite[Section~3,a--d]{BZtor}.
By~\cite[Theorem~3.7]{BZtor}, its exterior derivative as a current is given by
\begin{equation*}
  d\psi\bigl(\nabla^{S^+W},g^{S^+W}\bigr)
  =\pi^*e\bigl(\nabla^{S^+W}\bigr)-\delta_W\;,
\end{equation*}
where~$W\subset S^+W$ denotes the zero section.
Assume that~$\bar s\in\Gamma(S^+W)$ is transversal to the zero section as above.
Then by~\cite[Remark~3.8]{BZtor}, $\bar s^*\psi(\nabla^{S^+W},g^{S^+W})$
is again locally
integrable on~$W$ and satisfies
\begin{equation}\label{eq:psi}
  d\bigl(\bar s^*\psi\bigl(\nabla^{S^+W},g^{S^+W}\bigr)\bigr)
  =e\bigl(\nabla^{S^+W}\bigr)-\delta_{\bar s^{-1}(0)}\;,
\end{equation}
where the orientation of~$\bar s^{-1}(0)\subset W$ is determined
by the sign of~$\det d\bar s$.

Let~$D_M$ be the spin Dirac operator acting on~$\Gamma(SM)$,
and let~$B_M$ denote the odd signature operator
acting on~$\Omega^{\mathrm{ev}}(M)$.
Let~$\eta$ denote the Atiyah-Patodi-Singer $\eta$-invariant~\cite{APS1},
and write~$h(A)$ for the dimension of the kernel of an operator~$A$.

\begin{thm}\label{Thm.1}
For any metric $g$ and non-vanishing spinor field $s$ on a closed
7-manifold $M$,
	$$\nu(s)=2\int_Ms^*\psi\bigl(\nabla^{SM},g^{SM}\bigr)
		-24\,(\eta+h)(D_M)+3\,\eta(B_M)\in\Z/48\;.$$
\end{thm}

\begin{proof}
  Let~$\ahat(\nabla^{TW})$ and~$L(\nabla^{TW})$ be the Chern-Weil forms
  representing the $\ahat$- and the Hirzebruch $L$-class on~$W$,
  constructed from the Levi-Civita connection on~$(W,g^{TW})$.
  With our orientation convention for~$S^+W$,
  we have
  \begin{equation}\label{1.2}
    2e\bigl(\nabla^{S^+W}\bigr)
    =48\,\ahat\bigl(\nabla^{TW}\bigr)^{[8]}
    +e\bigl(\nabla^{TW}\bigr)-3\,L\bigl(\nabla^{TW}\bigr)^{[8]}
    \in\Omega^8(W)\;.
  \end{equation}
  This follows from~\cite[equation~(1)]{CrN} and the naturality of Chern-Weil forms.
  Let~$D^+_W$ %
  denote the spin Dirac operator on~$W$,
  and let~$\ind_{\mathrm{APS}}(D^+_W)$ denote its index under
  APS boundary conditions, see~\cite[(2.3)]{APS1}.
  By the Atiyah-Patodi-Singer index theorem~\cite[Theorems~4.2 and~4.14]{APS1}
  and equation~\eqref{eq:psi} for the Mathai-Quillen current,
  \begin{align}
    \begin{split}\label{1.3}
      \ind_{\mathrm{APS}}\bigl(D^+_W\bigr)
      &=\int_W\ahat\bigl(\nabla^{TW}\bigr)^{[8]}
	-\frac{\eta+h}2(D_M)\quad\in\Z\;,\\
      \sigma(W)
      &=\int_WL\bigl(\nabla^{TW}\bigr)^{[8]}-\eta(B_M)\quad\in\Z\;,\\
      n(\bar s)
      &=\int_W\delta_{\bar s^{-1}(0)}
      =\int_We\bigl(\nabla^{S^+W}\bigr)
      -\int_Ms^*\psi\bigl(\nabla^{SM},g^{SM}\bigr)\quad\in\Z\;.
    \end{split}
  \end{align}
  Definition~\ref{def:nu} together with the Gau\ss-Bonnet-Chern theorem
  and~\eqref{1.2}, \eqref{1.3} gives
  \begin{align*}
    \nu(s)
    &\equiv\chi(W)-3\,\sign(W)+48\,\ind_{\mathrm{APS}}(D^+_W)-2n(\bar s)&&\mod 48\\
    &=2\int_Ms^*\psi\bigl(\nabla^{SM},g^{SM}\bigr)
	-24\,(\eta+h)(D_M)+3\,\eta(B_M)\;.&&\qedhere\\
  \end{align*}
\end{proof}

\subsection{The extended \texorpdfstring{$\nu$}{nu}-invariant}\label{Abs1.b}

We first restrict attention to the special case of metrics with holonomy
contained in~$G_2$. Because the defining spinor~$s\in\Gamma(SM)$
is associated with the $G_2$-principal bundle,
this immediately implies~$\nabla^{SM}s=0$.

\begin{lem}\label{Lemma3.3}
  If~$s$ is parallel, then
	$$\int_Ms^*\psi\bigl(\nabla^{SM},g^{SM}\bigr)=0\;.$$
\end{lem}

\begin{proof}
  Let~$\bigl(\widehat{SM},\nabla^{\widehat{SM}},g^{\widehat{SM}}\bigr)\to M$
  be  an isomorphic copy of~$SM\to M$.
  We regard the curvature~$R^{\widehat{SM}}$
  as an element of~$\Omega^2(M;\Lambda^2\widehat{SM})$.
  Let~$\widehat Y \in\Gamma(\pi^*\widehat{SM})$ denote the tautological section
  of~$\pi^*\widehat{SM}\to SM$.
  Then~$\nabla^{\pi^*\widehat{SM}}\widehat Y$ projects a vector on~$SM$
  to its vertical component,
  regarded as an element of~$\widehat{SM}$.
  The Berezin integral
	$$\int^B\colon\Omega^\bullet(SM;\pi^*\Lambda^\bullet\widehat{SM})
	\to\Omega^\bullet(SM)$$
  is defined as a certain constant multiple of the top degree component
  in~$\pi^*\Lambda^\bullet\widehat{SM}$, see~\mbox{\cite[(3.1)]{BZtor}}.
  The precise value of the constant will not matter in the following.

  By~\cite[Definition~3.6]{BZtor}, the Mathai-Quillen current is given as
  \begin{equation*}
    \psi\bigl(\nabla^{SM},g^{SM}\bigr)
    =\int_0^\infty\int^B\frac{\widehat Y}{2\sqrt t}
	\,e^{-\pi^*R^{\widehat{SM}}+\sqrt t\,\nabla^{\pi^*\widehat{SM}}\widehat Y+t\|\widehat Y\|^2}\,dt\;.
  \end{equation*}
  For the pullback by~$s\in\Gamma(SM)$,
  we consider the section~$\hat s=s^*\widehat Y\in\Gamma(\widehat{SM})$.
  Then
  \begin{equation*}
    s^*\psi\bigl(\nabla^{SM},g^{SM}\bigr)
    =\int_0^\infty\int^B\frac{\hat s}{2\sqrt t}
	\,e^{-R^{\widehat{SM}}+\sqrt t\,\nabla^{\widehat{SM}}\hat s+t\|\hat s\|^2}\,dt\;.
  \end{equation*}

  If~$s$ is parallel, then~$\nabla^{\widehat{SM}}\hat s=0$.
  As a consequence,
  the exponential expression has even degree in~$\Lambda^\bullet\widehat{SM}$.
  The additional~$\hat s$ makes the degree in~$\Lambda^\bullet\widehat{SM}$ odd.
  Because~$\rk\widehat{SM}=8$ is even, the Berezin integral vanishes entirely.
\end{proof}

If the spinor~$s$ is parallel,
then the holonomy of~$(M,g)$ is contained in~$G_2$.
Then~$M$ is Ricci flat and the scalar curvature vanishes, too.
By the Schr\"odinger-Lichnerowicz formula,
\begin{equation*}
  D_M^2=\nabla^{SM,*}\nabla^{SM}\;,
\end{equation*}
so~$\ker(D_M)$ consists entirely of parallel spinors if~$M$ is compact.

The spinor representation of $G_2$ is isomorphic to a direct sum of the
(7-dimensional) vector representation and a rank one trivial part.
Therefore $SM \cong TM \oplus \trivr \cong T^*M \oplus \trivr$.
Indeed, Clifford multiplication with~$s$ defines a parallel isomorphism
from~$TM$ to the subbundle of~$SM$ that is perpendicular to~$s$.
On a closed Ricci-flat manifold, a 1-form is parallel if and only if it is
harmonic. Thus, for a closed manifold with holonomy contained in $G_2$
\begin{equation}
\label{eq:hDM} h(D_M) = 1 + b_1(M) .
\end{equation}
If the holonomy is exactly $G_2$ then $\pi_1 M$ is finite, so $h(D_M) = 1$,
but more generally we at least have that $h(D_M)$ is a topological invariant
if the holonomy is contained in~$G_2$. %
This motivates the following definition.

\begin{dfn}\label{MainDef}
  Let~$(M,g)$ be a Riemannian 7-manifold with holonomy contained in~$G_2$.
  Then the extended %
  $\nu$-invariant is given by
	$$\bar\nu(M,g)=%
		-24\,\eta(D_M)+3\,\eta(B_M)\in\Z\;.$$
\end{dfn}

\begin{prop}\label{Prop3.5}
  The extended $\nu$-invariant has the following properties.
  \begin{enumerate}
  \item\label{3.5.0}
    It is a spin diffeomorphism invariant of Riemannian manifolds
    with holonomy in~$G_2$.
  \item\label{3.5.1}
    It is locally constant on
    the moduli space of metrics with holonomy in~$G_2$.
  \item\label{3.5.2}
    If~$(M,g)$ admits an orientation reversing isometry,
    then~$\bar\nu(g)=0$.
  \end{enumerate}
\end{prop}

\begin{proof}
  Property~\ref{3.5.0} is clear by construction.

  For~\ref{3.5.1}, we use that~$h(D_M)$ is
  constant on the moduli space of metrics with holonomy in~$G_2$,
  so no eigenvalue of~$D_M$ can change sign over this moduli space.
  Then~$\eta(B_M)$ and~$\eta(D_M)$ are continuous in~$g$ by~\eqref{1.3},
  hence~$\bar\nu(M,g)$ is locally constant.

  For~\ref{3.5.2}, we use that both $\eta$-invariants vanish because
  an orientation reversing isometry ensures that the spectra of~$D_M$
  and~$B_M$ are symmetric.
\end{proof}

Combining Theorem \ref{Thm.1}, Lemma \ref{Lemma3.3} and \eqref{eq:hDM} we find
that for any metric $g$ with holonomy contained in $G_2$ and any spinor
$s$ that is parallel with respect to $g$ (equivalently any torsion-free
\gtstr{} compatible with $g$)
\begin{equation}
\nu(s) = \bar \nu(g) + 24(1 + b_1(M)) \mod 48 .
\end{equation}

\subsection{Homogeneous examples}\label{Abs1.c}

Leaving the world of metrics with holonomy in~$G_2$ for the moment,
we can define a version of the extended $\nu$-invariant for more general
\gtstr s. It is most meaningful in the case when the metric has positive
scalar curvature, implying $h(D_M) = 0$ by the Schr\"odinger-Lichnerowicz formula.

\begin{dfn}\label{dfn:nubar}
  Let~$(M,g,s)$ be a Riemannian 7-manifold with a nowhere vanishing
  spinor~$s$.
  Then the extended %
  $\nu$-invariant is defined as
	$$\bar\nu(M,g,s)=2\int_Ms^*\psi\bigl(\nabla^{SM},g^{SM}\bigr)
		-24\,\eta(D_M)+3\,\eta(B_M)\in\Z\;.$$
\end{dfn}

\begin{rmk} We distinguish three cases.
  \begin{enumerate}
  \item\label{Case1} If~$s$ is parallel,
    this is exactly Definition~\ref{MainDef} (using Lemma \ref{Lemma3.3}).
  \item\label{Case2} If~$(M,g)$ has positive scalar curvature,
    then~$\bar\nu(M,g,s)$ is invariant under deformations of~$g$ and~$s$
    as long as positive scalar curvature is preserved.
    In this case, $\bar\nu\mmod48$ equals the invariant~$\nu(s)$
    from Definition \ref{def:nu}.
    The situation is superficially similar to case~\ref{Case1}.
    However, there is no direct link between the spinor and the metric.
    A similar extension of the Eells-Kuiper invariant was considered
    by Kreck and Stolz in~\cite{KS3}.
  \item\label{Case3} Without any additional assumption on~$(M,g,s)$,
    the number~$\bar\nu(M,g,s)$ can jump by multiples of~$24$ under continuous
    deformations of $g$ and is hence less powerful than~$\nu(s)$. However, we
    still have $\nu(s) = \bar \nu(g,s) + 24h(D_M) \mmod 48$.
  \end{enumerate}
\end{rmk}

\begin{ex}
Consider the Berger space~$M=SO(5)/SO(3)$
with its normal homogeneous metric~$g$,
which is of positive scalar curvature.
Its diffeomorphism type has been determined in~\cite{GKS}
with the help of a homogeneous $G_2$-structure.
The $\eta$-invariants~$\eta(D_M)$ and~$\eta(B_M)$
have almost been computed in~\cite[Corollary~2.5]{GKS},
in the following sense.
Let~$\widetilde L$ denote
the Chern-Simons form for the $L$-class characterised
by~$d\widetilde L(\nabla^0,\nabla^1)
=L(\nabla^1)-L(\nabla^0)$,
and let~$\skew6\widetilde{\widehat A}$ be the analogous form for~$\ahat$,
then
\begin{align*}
	\eta(D_M)
	&=-\frac{12923}{2\;3^2\;5^6}
		+2\int_M\skew6\widetilde{\widehat A}\left(\nabla^0,\nabla^{TM}\right)\;,\\
		\text{and }\quad
	\eta(B_M)
	&=-\frac{4817}{3^2\;5^6}
		+\int_M\widetilde L\left(\nabla^0,\nabla^{TM} \right)\;.
\end{align*}
Here, $\nabla^{TM}$ is the Levi-Civita connection with respect to the normal
homogeneous metric on~$M$, and~$\nabla^0$ is the reductive connection.
Note that the factor of~$2$ in front of the correction term in the first line
is missing in~\cite{GKS},
and note also that in degree~$7$, the class~$\skew3\widetilde{\widehat L}$
in~\cite{GKS} agrees with~$\widetilde L$.

The homogeneous $G_2$-structure corresponds
to a section~$s\in\Gamma(SM)$
that is parallel with respect to the reductive connection~$\nabla^0$,
so by Lemma~\ref{Lemma3.3}, we get
	$$s^*\psi\bigl(\nabla^0,g^{SM}\bigr)=0\;.$$
By the variation formula for Mathai-Quillen currents~\cite{BZtor}
and~\eqref{1.2},
this implies
\begin{equation*}
  2s^*\psi\bigl(\nabla^{SM},g^{SM}\bigr)
  =2\tilde e\left(\nabla^0,\nabla^{SM}\right)
  =48\skew6\widetilde{\widehat A}\left(\nabla^0,\nabla^{TM}\right)^{[7]}
	-3\widetilde L\left(\nabla^0,\nabla^{TM}\right)^{[7]}\;.
\end{equation*}
Note that the variation of the Euler form of~$TM$ does not appear here
because~$TM$ is odd-dimensional.
  Hence, the homogeneous $G_2$-structure on the Berger space satisfies
	$$\bar\nu(M,g,s)
  =24\cdot\frac{12923}{2\;3^2\;5^6}-3\cdot\frac{4817}{3^2\;5^6}
  =1\;.$$
\end{ex}

\begin{ex}
On the round sphere,
we can construct $G_2$-structures using two kinds of Killing spinors~$s_\pm$.
Both give rise to homogeneous structures~$S^7=\Spin(7)/G_2$.
These examples have been discussed in~\cite[Example~1.14]{CrN}.
We fix the round metric~$g$ on~$S^7$,
for which~$\eta(D_{S^7})=\eta(B_{S^7})=0$.
It is not hard to check that~$\bar\nu(S^7,g,s_\pm)=\pm1$.

\end{ex}

\pagebreak[2]

\section{Twisted Connected Sums}\label{Kap2}

We recall the twisted connected sum construction of~\cite{Kovalev, CHNP}
and describe extra-twisted connected sums, see~\cite{xtcs}.
These manifolds are $G_2$-holonomy manifolds glued together from two
7-manifolds with holonomy~$SU(3)$ and an asymptotically cylindrical end,
where the gluing is ``twisted'' in such a way that the resulting manifold
can be equipped with a metric of holonomy~$G_2$.

\subsection{ACyl Calabi-Yau 3-folds with automorphisms}\label{subsec:acyl}

Let us first describe the pieces to be used in the gluing construction.

Let~$\K$ be a K3 surface. By a \hk structure on $\K$ we mean a triple
of closed 2-forms~$\omega^I, \omega^J, \omega^K$ such that
\[  (\omega^I)^2 = (\omega^J)^2 = (\omega^K)^2\text{ vanishes nowhere and }
\omega^I \wedge \omega^J = \omega^J \wedge \omega^K = \omega^K \wedge \omega^I = 0. \]
For such a \hk triple, there exists a Ricci-flat metric $g^\K$ and
integrable complex structures $I, J, K$ such that $g^\K$ is a Kähler metric
with Kähler form $\omega^I$, $\omega^J$ and $\omega^K$ respectively.
Also, $\omega^J + i \omega^K$ is a holomorphic 2-form with respect to $I$.

For~$\lnn>0$,
let~$S^1_{\lnn}=\R/\lnn\Z$ denote a circle of length~$\lnn$.
On~$\R_+\times S^1_{\lnn}$, we define a complex structure~$I_\C$
such that~$I_\C\del_t=\del_\anglen$ and~$I_\C\del_\anglen=-\del_t$,
where~$t$ and~$\anglen$ are coordinates on~$\R_+$ and~$S^1_{\lnn}$, respectively.
The corresponding K\"ahler form is~$dt\wedge d\anglen$.

By~\cite[Definition~3.3]{CHNP}, an asymptotically cylindrical Calabi-Yau
$3$-fold is a complex $3$-dimensional Calabi-Yau
manifold~$(V,g^V,I^V,\omega^V,\Omega^V)$ with a compact subset~$K$
such that $V\setminus K\cong\R_+\times S^1_{\lnn}\times\K$,
and such that there exists~$c>0$
such that on~$\R_+\times S^1_{\lnn}\times\K$ as~$t\to\infty$,
\begin{equation}
  \begin{aligned}\label{2.2}
    \omega^V&=dt\wedge d\anglen+\omega^I+d\alpha&&
    \text{for some~$\alpha$ with }\norm\alpha_{C^k}=O\bigl(e^{-c t}\bigr)\;,\\
    \Omega^V&=(d\anglen-i\,dt)\wedge(\omega^J + i \omega^K)+d\beta&&
    \text{for some~$\beta$ with }\norm\beta_{C^k}=O\bigl(e^{-c t}\bigr)\;,\\
    g^V&=\omega^V(\punkt,I^V\punkt)\;,
  \end{aligned}
\end{equation}
for all~$k\ge 0$, for some \hk structure $(\omega^I, \omega^J, \omega^K)$
on $\K$.
Here, $\norm\punkt_{C^k}$ is taken with respect to the background
metric on~$\R_+\times S^1_{\lnn}\times\K$ obtained by putting~$\alpha=\beta=0$,
\ie $dt^2 + du^2 + g^\K$.
We refer to~$S^1_{\lnn}\times\K$ as the cross section at infinity.

Fix~$\lnx>0$
and put~$\widetilde M=V\times S^1_{\lnx}$.
Let~$\anglex$ be the coordinate of the new ``external'' $S^1_{\lnx}$.
By \mbox{\cite[equation~(2.38)]{CHNP}},
the manifold~$\widetilde M$ carries a %
$G_2$-holonomy metric
with associated $3$-form
\begin{equation}\label{2.3}
  \phy=d\anglex\wedge\omega^V+\Re\Omega^V\;.
\end{equation}

In order to accommodate extra-twisted connected sums, %
we put~$M=\widetilde M/\Gamma$,
where~$\Gamma\cong\Z/k$.
We assume that~$\Gamma$ acts freely by rotations
on the external circle~$S^1_{\lnx}$,
preserves the Calabi-Yau structure on~$V$,
and induces a trivial action on the K3 surface~$\Sigma$
and a free action on the interior circle~$S^1_\lnn$.
In particular, the quotient~$M$ is smooth, and has an asymptotically
cylindrical end with cross section at infinity isometric to
\begin{equation}\label{eqn:X}
  X=\bigl((S^1_{\lnn}\times S^1_{\lnx})/\Gamma\bigr)\times\K\;.
\end{equation}
Note that~$(S^1_{\lnn}\times S^1_{\lnx})/\Gamma$ is again a two-torus,
on which~$(\del_\anglen,\del_\anglex)$ still defines an ortho\-normal frame
of tangent vectors.
The forms~$\alpha$ and~$\beta$ above can be chosen $\Gamma$-invariant.
One can construct examples of asymptotically cylindrical Calabi-Yau manifolds
with an action of~$\Gamma\cong\Z/2$ starting from Fano or weak Fano 3-folds of
index 2, as outlined in Examples \ref{ex:cubic_block} and \ref{ex:sextic_block}.
More general examples with~$k\ge 3$ are considered in~\cite{GN}.

\subsection{The Gluing construction}\label{subsec:glue}
The extra-twisted connected sum construction involves the following data.
\begin{itemize}
\item
Two asymptotically cylindrical Calabi-Yau manifolds~$V_+$ and~$V_-$,
with asymptotic cross-sections $S^1_{\lnn_\pm} \times \kd_\pm$
(where $\kd_\pm$ is a K3 surface and $\lnn_\pm$ is the circumference lengths
of the ``internal'' $S^1$ factor), admitting an action of
$\Gamma_\pm = \Z/k_\pm$ by automorphisms as above.
\item An angle~$\thet\in(0,\pi)$, which we will refer to as the \emph{gluing angle}.
\item A \emph{\hk rotation} $\hkr : \kd_+ \to \kd_-$, \ie 
the \hk structures on~$\K_+$ and~$\K_-$ are related by
\begin{equation}
\label{2.4}
\begin{aligned}
\hkr^*\omega^K_- &= - \omega^K_+ \\
\hkr^*(\omega^I_- + i \omega^J_-) &= e^{i\vartheta} (\omega^I_+ - i\omega^J_+) ;
\end{aligned}
\end{equation}
see also~\cite[Def.~3.10]{CHNP} for~$\thet=\frac\pi2$.
\item The lengths of the exterior circles~$\lnx_+, \lnx_- > 0$.
\item An orientation-reversing isometry $T^2_+ \to T^2_-$, where 
$T^2_\pm$ is the torus $(S^1_{\lnn_\pm}\times S^1_{\lnx_\pm})/\Gamma_\pm$,
such that the orthogonal frames are related by
\begin{align}
  \begin{split}\label{2.5}
    \del_{\anglex_-}&=\cos\thet\,\del_{\anglex_+}+\sin\thet\,\del_{\anglen_+}\;,\\
    \del_{\anglen_-}&=\sin\thet\,\del_{\anglex_+}-\cos\thet\,\del_{\anglen_+}\;
  \end{split}
\end{align}
(see Figures~\ref{fig:1/4}, \ref{fig:1/6} in \S \ref{Kap4} for illustrations where $\thet = \frac\pi4$ or
$\frac\pi6$).
\end{itemize}
Given this data, we construct~$(M_{\pm},g^{TM_\pm},\phy_{\pm})$ as above.
Let~$\rho\colon\R\to[0,1]$ be a smooth cutoff function
such that~$\rho(x)=0$ for~$x\le 0$ and~$\rho(x)=1$ for~$x\ge 1$.
Let~$\alpha_\pm$, $\beta_\pm\in\Omega^\bullet(V_\pm)$ be as in~\eqref{2.2}.
Let~$\ell\gg 1$, and put
\begin{align}
  \begin{split}\label{2.6}
    \omega^{V_\pm}_\ell&=\omega^{V_\pm}-d\bigl(\rho(t_\pm{-}\ell)\alpha_\pm\bigr)\;,\\
    \Omega^{V_\pm}_\ell&=\Omega^{V_\pm}-d\bigl(\rho(t_\pm{-}\ell)\beta_\pm\bigr)\;,
  \end{split}
\end{align}
as in~\cite[equations~(3.8)]{CHNP}.
Assuming that~$\alpha_\pm$, $\beta_\pm$ are $\Gamma_\pm$-invariant,
we define closed, but not torsion free,
$G_2$-structures
\begin{equation}\label{2.6a}
  \phy_{\pm,\ell}=dv_\pm\wedge\omega^{V_\pm}_\ell
  +\Re\Omega^{V_\pm}_\ell
\end{equation}
on~$M_{\pm}$ as in equation~\eqref{2.3}.
In particular,
\begin{equation}\label{2.7}
  \phy_{\pm,\ell}|_{(\ell+1,\ell+3)\times\X} %
  =d\anglex_\pm\wedge \omega^I_\pm
	+d\anglen_\pm\wedge \omega^J_\pm 
	+dt_\pm\wedge \omega^K_\pm 
        +dt_\pm\wedge d\anglen_\pm\wedge d\anglex_\pm\;,
\end{equation}
see~\cite[eq.~(3.12)]{CHNP}.

We may identify~$(\ell+1,\ell+3)\times\bigl((S^1_{\zeta_\pm}\times S^1_{\xi_\pm})/\Gamma_\pm\bigr)\times\K_\pm$
with~$(-1,1)\times X$
using the isometries of~$\K_+$ and~$\K_-$
and of~$T^2_+$ and~$T^2_-$ above,
such that~$t_++t_-=2\ell+4$.
Hence,
let~$V_{\pm,\ell}$ denote the
manifold~$V_\pm\setminus((\ell+2,\infty)\times S^1_{\lnn_\pm}\times\K_\pm)$,
and put~$\widetilde M_{\pm,\ell}=V_{\pm,\ell}\times S^1_{\lnx_\pm}$
and~$M_{\pm,\ell}=\widetilde M_\ell/\Gamma_\pm$.
Then~$M_{\pm,\ell}$ is a manifold with boundary~$X$.
We now define
\begin{equation}\label{eqn:Ml}
  M_\ell=M_{-,\ell}\cup_{X}M_{+,\ell}\;.
\end{equation}
It then follows from~\eqref{2.4}--\eqref{2.7} that~$\phy_{+,\ell}$
and~$\phy_{-,\ell}$ extend to a smooth and closed (but not torsion-free)
$G_2$-structure~$\phy_\ell$ on~$M_\ell$.

As coordinate on~$(-\ell-2,\ell+2)\times X$ we choose
\begin{equation}\label{tCoord}
  t \; =\; t_--\ell-2 \; = \; \ell+2-t_+\;.
\end{equation}
Then~$t$ is an inward normal coordinate for~$M_+$.
The compatible orientation on~$X$ is given by combining the usual orientation
of the K3 surface~$\K$ with the orientation of~$T^2$ given by
the two parallel orthonormal frames~$(\del_{\anglen_-},\del_{\anglex_-})$
and~$(\del_{\anglex_+},\del_{\anglen_+})$ of~\eqref{2.5}.

Kovalev~\cite[Theorem~5.34]{Kovalev}
proves that there is a torsion free $G_2$-structure~$\bar\phy_\ell$
in the cohomology class of~$\phy_\ell$ if~$\ell$ is sufficiently large,
in the case that~$\thet=\frac\pi2$ and~$\Gamma_+=\Gamma_-=\{\id\}$ are trivial.
The same argument holds in the more general case.

\begin{thm}\label{Thm2.1}
  For~$\ell$ sufficiently large,
  there exists a torsion free $G_2$-structure~$\bar\phy_\ell$
  in the cohomology class of~$\phy_\ell$
  such that for each~$k_0$,
  there exists a constant~$c$ such that for all~$k\le k_0$,
  \begin{equation*}
    \norm{\phy_\ell-\bar\phy_\ell}_{C^k}\le e^{-c\ell}
  \end{equation*}
  with respect to the Riemannian metric associated to~$\phy_\ell$.
\end{thm}

\begin{proof}
  Let~$\bar\phy_\ell$ denote the torsion-free $G_2$-structure
  in the cohomology class of~$\phy_\ell$,
  which exists by~\cite[Theorem~5.34]{Kovalev}.
  Now, the theorem follows by bootstrapping
  using~\cite[Proposition~5.32]{Kovalev}
  with~$\bar\Theta=\Theta(\phy_\ell)-\Theta(\bar\phy_\ell)$,
  see also~\cite[p. 303]{Joyce}.
\end{proof}

\subsection{Matching and configurations}

Theorem \ref{Thm2.1} raises the question of how to find examples of the data
needed to apply it---we call this the \emph{matching problem}.
A further question is how to compute topological
properties of the resulting 7-manifolds. The notion of a configuration
of polarising lattices of ACyl Calabi-Yau 3-folds turns out to be crucial
to both questions.

\begin{dfn}\label{NpmDef}
For an ACyl Calabi-Yau 3-fold $V$ with asymptotic cross-section
$S^1 \times \kd$, call the image $N$ of the restriction map
$H^2(V;\Z) \to H^2(\kd;\Z)$ equipped with the restriction of the intersection
form of $\kd$ the \emph{polarising lattice} of $V$.
\end{dfn}

If $V$ has full holonomy $SU(3)$ then $N \subset H^{1,1}(\kd)$, so that
$\kd$ is an ``$N$-polarised'' K3 surface. 
Since the polarising lattice contains a Kähler form and is also orthogonal
to the real and imaginary parts of a holomorphic 2-form, it must be
non-degenerate of signature $(1, \rk N - 1)$.

Up to isometry, there exists a unique even non-singular lattice $\latc$ of
signature $(3,19)$, so $H^2(\kd;\Z)$ is isometric to $\latc$ for any
K3 surface $\kd$. Thus we can consider the polarising lattice of an ACyl
Calabi-Yau 3-fold as a sublattice of $\latc$, well-defined up to the action
of $O(\latc)$.
Given a pair of ACyl Calabi-Yau 3-folds $V_\pm$ and a \hk rotation
$\hkr : \kd_+ \to \kd_-$, we can 
instead consider the \emph{pair} of sublattices $N_+, N_- \subset \latc$,
and thus associate to $\hkr$ a well-defined configuration in the following
sense.

\begin{dfn}\label{ConfDef}
Given a pair of lattices $N_+, N_-$, a \emph{configuration} is a pair of
embeddings of $N_+$ and $N_-$ into the K3 lattice $\latc$, where two pairs are
considered equivalent if they are related by the action of $O(\latc)$.
\end{dfn}

Much of the topology of a twisted connected sum can be computed from data
about the ACyl Calabi-Yaus $V_\pm$ individually together with the configuration,
\eg the cohomology is easily computed using Mayer-Vietoris.
The following property of the configuration also affects the value of
$\bar \nu$.

\begin{dfn}\label{def:angles}
Given a configuration $N_+, N_- \subset \latc$, let $A_\pm : \latc_\R \to \latc_\R$
denote the reflection of $\latc_\R := \latc \otimes \R$ in $N_\pm$ (with respect to the
intersection form of $\latc_\R$; this is well-defined since $N_\pm$ is non-degenerate).
Suppose that
\begin{equation}
\label{eq:preserve}
\begin{aligned}
&A_+ \circ A_- \textrm{ preserves some decomposition }\latc_\R = \latc^+ \oplus \latc^-\\[-1mm]
&\textrm{ as a sum of positive and negative-definite subspaces.}
\end{aligned}
\end{equation}
Then the \emph{configuration angles} are the arguments (in $(-\pi,\pi]$)
$\alpha^+_1, \alpha^+_2, \alpha^+_3$ and
$\alpha^-_1, \ldots, \alpha^-_{19}$ of the eigenvalues of the restrictions
$A_+ \circ A_- : \latc^+ \to \latc^+$ and $A_+ \circ A_- : \latc^- \to \latc^-$ respectively.
\end{dfn}

From the gluing angle and the configuration angles,
we can define the terms on the right hand side of Corollary~\ref{MainThm},
which completely determine the extended $\nu$-invariant of an extra-twisted
connected sum with $k_\pm \leq 2$.

\begin{dfn}\label{def2.6}
  Let~$\thet$ be the gluing angle,
  let~$\rho=\pi-2\thet$, and let %
  $\alpha_1^-$, \dots, $\alpha_{19}^-\in(-\pi,\pi]$ be the configuration
  angles.
  Then put
  \begin{equation*}
    m_\rho(\latc;N_+,N_-)
    =\Bigl(\#\bigl\{\,j
		\bigm|\alpha_j^-\in\{\pi-\abs\rho,\pi\}\,\bigr\}-1
        +2\,\#\bigl\{\,j
		\bigm|\alpha_j^-\in(\pi-\abs\rho,\pi)\,\bigr\}\Bigr)\, \Sign\rho \;.
  \end{equation*}
  If $\rho = 0$ then set $m_0 = 0$.
\end{dfn}

Given a pair of ACyl Calabi-Yau 3-folds $V_\pm$, there is in general no reason
to expect there to exist any \hk rotation between their asymptotic K3s.
On the other hand, if we want to understand the topology of the resulting
\gtmfd s, we don't need to control the actual Calabi-Yau structures, but only
the topology of the underlying ACyl Calabi-Yau manifold
(along with the configuration of the \hk rotation).
It is therefore fruitful to phrase the problem as follows. 

\begin{matchingproblem}
\label{match}
Given $\thet$, a pair of deformation families of ACyl Calabi-Yau 3-folds
and a configuration of their polarising lattices $N_\pm$, does there exist some
pair of members with a $\thet$-\hk rotation compatible with that configuration?
\end{matchingproblem}

For a positive answer, there are various necessary conditions on the
configuration. Most relevant for us is that the
condition \eqref{eq:preserve} in Definition \ref{def:angles} must be satisfied.
This is because the metric of the \hk structure
(which is preserved by the \hk rotation) defines a splitting of $\latc_\R$
into its self-dual and anti-self-dual parts, and $N_\pm\otimes\R$
splits as a sum of the span
of~$[\omega^I_\pm]$ (which is self-dual) and the anti-self-dual part
of $N_\pm\otimes\R$. Thus $A_{\pm}$ each preserve the splitting of $\latc_\R$,
so certainly $A_+ \circ A_-$ does too.
Moreover, \eqref{2.4} implies that
\begin{equation}
\label{eq2:+rotation}
\{ \alpha^+_1, \alpha^+_2, \alpha^+_3\} = \{0, 2\thet, -2\thet\} .
\end{equation}
Strategies for producing sufficient conditions for solving the matching
problem are discussed in~\cite[Section~6]{CHNP} and \cite[Section 5]{CrN3}
for the case of rectangular twisted connected sums
(\ie~${\thet=\frac\pi2}$, $\Gamma_-=\Gamma_+$ trivial),
and in~\cite[Section 6]{xtcs} for extra-twisted connected sums. We describe some of the
resulting examples below.

\begin{rmk}
\label{rmk:identify_angles}
Here is a way to identify the non-zero angles in
Definition \ref{def:angles} that is often convenient in examples.
Let $\pi_\pm : \latc \to N_\pm$ be the %
orthogonal projection to the non-degenerate sub\-lattice~$N_\pm$.
Then the restriction of $\pi_\pm \circ \pi_\mp$ to $N_\pm$ is self-adjoint.
Elements of the kernel of these operators are $\pi$-eigenvectors
of $A_+ \circ A_-$.
On the other hand, if $\phi \in (0,\pi)$ then $v \in N_+$ is a
$\big(\!\cos \frac\phi2 \big)^2$-eigenvector of $\pi_+\pi_-$ if and only if
$\pi_- v$ is a $\big(\!\cos \frac\phi2 \big)^2$-eigenvector of~$\pi_-\pi_+$

In summary, each of the 22 configuration angles that equals $\pi$ contributes
an eigenvalue 0 to one of $\pi_+\pi_-$ and $\pi_-\pi_+$, while each angle pair
$\phi, -\phi$ with $\phi \in (0,\pi)$ contributes an eigenvalue
$\big(\!\cos \frac\phi2 \big)^2$ to each of $\pi_+\pi_-$ and~$\pi_-\pi_+$.
\end{rmk} 

\section{Examples}\label{Kap4}

We outline some examples of extra-twisted connected sums described in detail
in~\cite[\S 8]{xtcs}, and compute their $\bar\nu$-invariants. The proofs of
Theorems \ref{Thm4.1} and \ref{Thm4.2} rely on Example \ref{ex:pi4rk2}
and~\ref{ex:pi6rk2}, respectively.
For each example we will indicate the pair of asymptotically cylindrical
Calabi-Yau manifolds used, describe the data of the matching required to
compute the configuration angles and then apply Corollary \ref{MainThm}.

In this section, we will use~$H^\bullet(\punkt)$ to refer to cohomology with
\emph{integer} coefficients.

To find diffeomorphisms between different (extra-) twisted connected sums
we use the following special case of the results of Wilkens
\cite[Theorem~2]{wilk1} and \cite[Theorem~1]{wilk2}
(\cf \mbox{\cite[\mbox{Theorems~4.22 and~4.25}]{CHNP}} and
\cite[Theorem~1.3]{CrN2}).

\begin{thm}
\label{thm:torfreeclass}
Smooth closed 2-connected 7-manifolds~$M$ with~$H^4(M)$ torsion-free are
classified up to almost-diffeomorphism by the isomorphism class of the pair
$(H^4(M), p_1(M))$,
or equivalently by~$b_3(M)$ and~$\div p_1(M)$,
which we define as the greatest integer dividing~$p_1(M)$.
Moreover, if~$\div p_1(M)$ is not divisible by 16 or 7, then the pair~$(b_3(M), \div p_1(M))$
determines~$M$ up to diffeomorphism.
\end{thm}

Here \emph{almost-diffeomorphism} means
a homeomorphism that is smooth away from a single point (so two smooth
manifolds are almost-diffeomorphic if one is diffeomorphic to the connected
sum of the other with an exotic sphere).

The problems of counting diffeomorphism classes in a given
almost-diffeomorphism class of smooth 7-manifolds and counting classes of
\gtstr s on a given smooth 7-manifold are closely related.
In the proof of Theorem \ref{Thm4.2} we will make use of the following
special case of~\mbox{\cite[Corollary 1.13]{CrN}}.

\begin{prop}
\label{prop:nu_classifies}
Let $M$ be a smooth closed 2-connected 7-manifold with $H^4(M)$ torsion-free.
If $\div p_1(M)$ divides 224 then there are precisely 24 classes of \gtstr s
on $M$ modulo homotopy and diffeomorphism, and they are distinguished by $\nu$.
\end{prop}
\subsection{Asymptotically cylindrical Calabi-Yaus with involution}

Let us first describe two families of examples of asymptotically cylindrical
Calabi-Yau manifolds with involution, constructed from Fano 3-manifolds~$X$
with index 2. That means $X$ is a closed complex 3-fold whose anticanonical
class~$-K_X$ is ample and even.

\begin{ex}[{\cite[Example 3.25(iii)]{xtcs}}]
\label{ex:cubic_block} 
Let~$X \subset \PP^4$ be a smooth cubic.
Let~$\K \subset X$ be a smooth section by a quadric, and let~$C \subset \K$ be
a section by a hyperplane. Let~$Y$ be the double cover of~$X$ branched over
$\K$, and~$Z$ the blow-up of~$Y$ in the curve~$C$. Then~$Z$ contains an
anticanonical divisor isomorphic to~$\K$, with trivial normal bundle, and
$V := Z \setminus \K$ admits asymptotically cylindrical Calabi-Yau metrics.
The branch-switching involution of~$Y$ generates an automorphism group
$\Gamma \cong \Z/2$
of~$V$ of the kind described in Subsection~\ref{subsec:acyl}.

The polarising lattice $N$ of~$V$ is the same as the Picard lattice of~$Y$,
\ie~$\Pic Y \cong H^2(Y)$ equipped with the bilinear form
$(D_1, D_2) \mapsto D_1 \cdot D_2 \cdot (-K_Y)$.
It is $N \cong (6)$, that is, it has a single generator with
square 6.
\end{ex}

\begin{ex}[{\cite[Example 3.25(i)]{xtcs}}]
\label{ex:sextic_block}
In the weighted projective space~$\PP^4(3, 2, 1, 1, 1)$,
consider a smooth sextic hyper\-surface~$X$,
such that the anti\-canonical section
$\K := \{X_1 = 0\}$ is smooth (where~$X_1$ is the weight 2 coordinate).
Let~$Y$ be the double cover of~$X$ branched over~$\K$. ($Y$~is a sextic
hypersurface in~$\PP^4(3, 1, 1, 1, 1)$; it is a double cover of~$\PP^3$
branched over a sextic surface.)

Let~$C \subset \K$ be the intersection with a hyperplane (of weight 1, like
$\{X_2 = 0\}$), and let $Z$ be the blow-up of~$Y$ in~$C$.
Then~$V := Z \setminus \K$
is an asymptotically cylindrical Calabi-Yau manifold with involution as above.
The Picard lattice of~$Y$ is~$N \cong (2)$.
\end{ex}

The examples we use to prove Theorems \ref{Thm4.1} and \ref{Thm4.2} both
rely on the following family of Calabi-Yau manifolds with involution,
whose topological properties turn out to be auspicious for
constructing extra-twisted connected sums without torsion in~$H^*(\punkt)$.

\begin{ex}[{\cite[Example 3.31]{xtcs}}] 
\label{ex:auspicious}
Let~$X'$ be a smooth sextic hypersurface 
$\PP^4(3, 2, 1, 1, 1)$ such that~$X'$ is tangent to~$\{X_1 = 0\}$ at
$p := (0 {:} 0 {:} 0 {:} 0 {:} 1)$.
Let~$X$ be the blow-up of~$X'$ at~$p$, and~$\K \subset X$ the proper transform
of the section~$\K' := \{X_1 = 0\} \cap X'$.
Generically~$p$ is an ordinary double point on~$\K'$, and~$\K$ is a smooth
section of~$-K_X$. Let~$Y$ be the double cover of~$X$ branched over~$\K$.
$Y$ has Picard lattice
$N \cong \bigl(\begin{smallmatrix}2&2\\ 2& 0 \end{smallmatrix}\bigr)$.

Blowing up~$Y$ in a curve and removing an anticanonical divisor isomorphic
to~$\K$ yields an asymptotically
cylindrical Calabi-Yau manifold with involution, and polarising lattice~$N$.
\end{ex}

\thefigures

\subsection{Extra-twisted connected sums with
\texorpdfstring{$\thet = \frac{\pi}{4}$}{angle pi/4}}

We now consider examples of extra-twisted connected sums with gluing angle
$\frac\pi4$. We use a pair of asymptotically cylindrical Calabi-Yau 3-folds
$V_+$ and $V_-$ with asymptotic cross-sections $S^1_{\lnn_\pm} \times \K_\pm$,
and require $V_+$ to have an automorphism group $\Gamma_+ \cong \Z/2$
like in Subsection~\ref{subsec:acyl}.
If we set $\lnx_+ = \lnn_+$,
then the torus factor in the boundary
of~$M_+ := (S^1_{\lnx_+}\times V_+)/\Gamma_+$ is a square torus, of side
length $\lnn_+/\sqrt 2$. If $\lnn_- = \lnx_- = \lnn_+/\sqrt 2$,
then there exists an isometry
$(S^1_{\lnx_+} \times S^1_{\lnn_+})/(\ant \times \ant)
\to S^1_{\lnx_-} \times S^1_{\lnn_-}$
with gluing angle $\thet = \frac\pi4$, illustrated in Figure \ref{fig:1/4}.

If we in addition have a \hk rotation $\kd_+ \to \kd_-$ with angle $\frac\pi4$
(in the sense of~\eqref{2.4}) then we have all the data required to
construct an extra-twisted connected sum of
$M_+$ and $M_- := S^1_{\lnx_-} \times V_-$.

\begin{ex}
\label{ex:pi4rk1}
Take $V_+$ to be an asymptotically cylindrical Calabi-Yau manifold
with involution from Example \ref{ex:sextic_block}, and $V_-$ an asymptotically
cylindrical Calabi-Yau manifold constructed from a blow-up of $\CP^3$
\cite[Row 1 of Table 1]{CHNP}. The polarising lattices are $N_+ \cong (2)$
and $N_- \cong (4)$. There is a solution to the matching problem \ref{match}
with gluing angle~$\thet = \frac\pi4$ and
configuration such that the intersection form on~$N_+ \oplus N_-$ is
\[ \begin{pmatrix} 2 & 2 \\2 & 4\end{pmatrix} . \]
Using Remark \ref{rmk:identify_angles}, it is easy to see that
the isometry~$A_{+}A_{-}$ of~$H^2(\K)$
rotates the two-dimensional span of $N_+$ and $N_-$ in $H^{2,+}(\K)$ by
$\frac \pi 2$ (this is the instance in this example of the two-dimensional
subspace rotated by $2\thet$ that always exists in $H^{2,+}(\K)$, see
\eqref{eq2:+rotation}),
and it fixes the orthogonal complement pointwise.
In particular, the configuration angles~$\alpha^\pm_i$ are given by
\begin{equation}
  \alpha^+_1=\tfrac\pi2\;,\qquad
  \alpha^+_2=-\tfrac\pi2\;,\qquad\text{and}\qquad
  \alpha^+_3=\alpha^-_1=\cdots=\alpha^-_{19}=0\;.
\end{equation}
We have~$\rho=\pi-2\thet=\tfrac\pi2$.
By Corollary~\ref{MainThm} and Definition \ref{def2.6}, we conclude that
\begin{equation*}
  \bar\nu(M,  g)=-39\;.
\end{equation*}
The resulting $G_2$-manifold is 2-connected with~$H^4(M) \cong \Z^{134}$
and~$\div p_1(M)=48$, see line 21 of \cite[Table 4]{xtcs}.
\end{ex}

The construction is also possible starting with~$\thet=\frac{3\pi}4$.
This would give~$\bar\nu$-invariant~$39$. The resulting manifold is
orientation-reversing isometric to a twisted connected sum of complex
conjugates of the two ACyl Calabi-Yau manifolds with gluing angle
$\frac{\pi}{4}$, see \cite[Remark 1.12]{xtcs}.

Example \ref{ex:pi4rk1} is not diffeomorphic to any 2-connected rectangular
twisted connected sum, since those always have odd~$b_3$
\cite[Theorem~4.8(iii)]{CHNP}.

\begin{ex}
\label{ex:pi4rk2}
Take~$V_+$ to be an asymptotically cylindrical Calabi-Yau manifold with
involution from Example \ref{ex:auspicious}.
The intersection form on the polarising lattice~$N_+$ is thus
given by~$\bigl(\begin{smallmatrix}2&2\\
2&0\end{smallmatrix}\bigr)$.

Let~$V_-$ be constructed from Entry 3 in the Mori-Mukai list of Fano 3-folds
of rank 2, see~\mbox{\cite[Table 3]{CrN3}}.
Then the intersection form on the polarising lattice~$N_-$
is described by~$\bigl(\begin{smallmatrix}4&2\\2&0\end{smallmatrix}\bigr)$.

The manifolds~$M_+$ and~$M_-$ can be glued with angle~$\thet=\frac\pi 4$
and a configuration such that~$N_+\cap N_-=0$, and
the intersection form on~$N_+\oplus N_-$ is given by
\begin{equation*}
  \begin{pmatrix}
    2&2&2&1\\2&0&2&0\\2&2&4&2\\1&0&2&0
  \end{pmatrix}\;,
\end{equation*}
where the first two coordinates correspond to a basis of~$N_+$,
and the last two to a basis of~$N_-$.
One can use Remark \ref{rmk:identify_angles} to check that the
isometry~$A_{+}A_{-}$ of~$H^2(\K)$ acts both on~$H^{2,+}(\K)$
and on~$H^{2,-}(\K)$ by rotating a two-dimensional plane by~$\frac\pi 2$
and fixing its respective orthogonal complement pointwise.
In particular, the configuration angles~$\alpha^\pm_i$ are given by
\begin{equation}
  \alpha^+_1=\alpha^-_1=\tfrac\pi2\;,\qquad
  \alpha^+_2=\alpha^-_2=-\tfrac\pi2\;,\qquad\text{and}\qquad
  \alpha^+_3=\alpha^-_3=\cdots=\alpha^-_{19}=0\;.
\end{equation}
By Corollary~\ref{MainThm}, we conclude that
\begin{equation*}
  \bar\nu(M, g)=-36\;.
\end{equation*}

It is computed in \cite[Example 8.1]{xtcs} that %
the $G_2$-manifold~$M$ is 2-connected, $H^4(M)\cong\Z^{97}$
and~$\div p_1(M)=4$.
\end{ex}

\begin{proof}[Proof of Theorem \ref{Thm4.1}]
According to~\cite[Row $b=74$ in Table 3]{CHNP},
there are rectangular twisted connected sums with the same topological
invariants, %
but with~$\bar\nu=0$ by Corollary~\ref{MainCor}.
By Theorem~\ref{thm:torfreeclass}, these manifolds are diffeomorphic
to the manifold from Example \ref{ex:pi4rk2}.
On the other hand, their $\bar\nu$-invariants are different mod~48, so
we have proved Theorem~\ref{Thm4.1}.
\end{proof}

In the above examples, all the angles $\alpha^\pm_j$ are $2\thet$, $-2\thet$
or 0.
The matching problem \ref{match} is easier to solve for such arrangements of
the polarising lattices, but it is not necessary to work exclusively with these
angles.

\begin{ex}
\label{ex:pi4mixed}
Let $V_+$ be an asymptotically cylindrical Calabi-Yau manifold with involution from
Example~\ref{ex:sextic_block}, which has polarising lattice $N_+ \cong (2)$.
Let $V_-$ be an asymptotically cylindrical Calabi-Yau manifold constructed from
entry 10 in the Mori-Mukai list of rank 2 Fano 3-folds,
\cf \mbox{\cite[Table 3]{CrN3}}. Its polarising lattice is
$N_- \cong \sm{0 & 4 \\ 4 & 8}$.
We can form a matching with angle $\thet = \frac\pi4$ where the
intersection form on $N_+ \oplus N_-$ is given by
\[ \begin{pmatrix} 2 & 1 & 3 \\ 1 & 0 & 4 \\ 3 & 4 & 8 \end{pmatrix} . \]
The action of $A_{+}A_{-}$ on $H^{2,+}(\K)$ is a rotation of a plane
by angle $\frac\pi2$ as usual, while on $H^{2,-}(\K)$ it is a reflection in a
hyperplane (the orthogonal complement of $N_-$ in $H^{2,-}(\K)$).
Thus
\begin{equation}
  \alpha^+_1=\tfrac\pi2\;,\qquad
  \alpha^+_2=-\tfrac\pi2\;,\qquad
  \alpha^-_1 = \pi\;,\qquad\text{and}\qquad
  \alpha^+_3=\alpha^-_2=\cdots=\alpha^-_{19}=0\;,
\end{equation}
and Corollary~\ref{MainThm} implies
\begin{equation*}
  \bar\nu(M, g)=-36\;.
\end{equation*}
It is computed in \cite[Example~8.7]{xtcs} that the $G_2$-manifold $M$ is
2-connected with $H^4(M) \cong \Z^{91}$ and $\div p_1(M) = 8$ 
\end{ex}

\begin{rmk}
\label{rmk:pi4mixed}
According to \cite[Table 4]{CrN3}, there are at least two rectangular twisted
connected sums with the same invariants, so this gives another example of
the same kind as Theorem~\ref{Thm4.1}.
\end{rmk}

\subsection{Extra-twisted connected sums with 
\texorpdfstring{$\thet = \frac{\pi}{6}$}{angle pi/6}}

For extra-twisted connected sums with $\thet = \frac\pi6$ we need a pair
of asymptotically cylindrical Calabi-Yau manifolds $V_+, V_-$ that both
have involutions.

We choose $\lnx_+ = \sqrt{3}\,\lnn_+$, so that the torus factor
$(S^1_{\lnx_+} \times S^1_{\lnn_+})/(-\id_{S^1}, -\id_{S^1})$ in the
asymptotic cross-section of
$M_+ :=  {(S^1_{\lnx_+} \times V_+)}/\{\id, \tau_+\}$ has hexagonal symmetry.
If we also have $\lnn_- = \sqrt{3}\, \lnx_- = \sqrt{3}\,\lnn_+$ then the
torus factor in the asymptotic cross-section of
$M_- :=  {(S^1_{\lnx_-} \times V_-)}/\{\id, \tau_-\}$ is isometric to that
in $M_+$, with $\thet = \frac \pi6$ in \eqref{2.5}. This is illustrated
in Figure~\ref{fig:1/6}.
If there is a \hk rotation between the K3 factors in the asymptotic
cross-sections in the sense of \eqref{2.4} then we can form an extra-twisted
connected sum $M$.

\begin{rmk}
In a similar way, one can produce simply-connected extra-twisted connected
sums with $\thet = \frac\pi3$ by setting $\lnn_- = \lnn_+$,
$\lnx_+ = \lnx_- = \sqrt3 \, \lnn_+$. However, the resulting manifolds tend
to have 3-torsion in $H^4$, making them less convenient for proving results
like Theorem~\ref{Thm4.1}.
\end{rmk}

\begin{ex}
\label{ex:pi6rk1}
Take $V_+$ and~$V_-$ to be asymptotically cylindrical Calabi-Yau manifolds
with involution from
Examples \ref{ex:sextic_block} and \ref{ex:cubic_block} respectively.
The polarising lattices are~$N_+ \cong (2)$ and~$N_- \cong (6)$.
We can find a $\thet = \frac\pi6$ matching with a configuration
where the intersection form on~$N_+ \oplus N_-$ is
\[ \begin{pmatrix} 2 & 3 \\3 & 6\end{pmatrix} . \]
The isometry~$A_{+}A_{-}$ of~$H^2(\K)$
rotates a two-dimensional plane in $H^{2,+}(\K)$ by~$2\thet=\frac{\pi}3$
in the usual way and fixes the orthogonal complement pointwise.
In particular, the configuration angles~$\alpha^\pm_i$ are given by
\begin{equation}
  \alpha^+_1=\tfrac{\pi}3\;,\qquad
  \alpha^+_2=-\tfrac{\pi}3\;,\qquad\text{and}\qquad
  \alpha^+_3=\alpha^-_1=\cdots=\alpha^-_{19}=0\;.
\end{equation}
We have~$\rho=\pi-2\thet=\tfrac{2\pi}3$.
By Corollary~\ref{MainThm}, we conclude that
\begin{equation*}
  \bar\nu(M,  g)=-51\;.
\end{equation*}
It is computed in \cite[Example 8.16]{xtcs} that the \gtmfd{} $M$ 
is 2-connected, $H^4(M) \cong \Z^{86}$ and~$\div p_1(M)=12$.
\end{ex}

Like Example \ref{ex:pi4rk1}, Example \ref{ex:pi6rk1} has even $b_3$ and is
therefore not diffeomorphic to any rectangular twisted connected sum.

\begin{ex}
\label{ex:pi6rk2}
Take both $V_+$ and $V_-$ from the family of asymptotically cylindrical
Calabi-Yau manifolds with involution in Example \ref{ex:auspicious},
so the polarising lattices are
given by~$N_\pm \cong \bigl(\begin{smallmatrix}2&2\\
2&0\end{smallmatrix}\bigr)$.

We can find a $\thet = \frac\pi6$ matching where~$N_+\cap N_-=0$,
and the intersection form on~$N_+\oplus N_-$ is given by
\begin{equation*}
  \begin{pmatrix}
    2&2&2&1\\2&0&1&2\\
    2&1&2&2\\1&2&2&0
  \end{pmatrix}\;.
\end{equation*}
Using Remark \ref{rmk:identify_angles}, we see that the
the isometry~$A_{+}A_{-}$ rotates one two-dimensional plane in each
of~$H^{2,+}(\K)$ and~$H^{2,-}(\K)$ by~$\frac{\pi}3$.
The configuration angles~$\alpha^\pm_i$ are thus 
\begin{equation}
  \alpha^+_1=\alpha^-_1=\tfrac{\pi}3\;,\qquad
  \alpha^+_2=\alpha^-_2=-\tfrac{\pi}3\;,\qquad\text{and}\qquad
  \alpha^+_3=\alpha^-_3=\cdots=\alpha^-_{19}=0\;.
\end{equation}
By Corollary~\ref{MainThm}, we conclude that
\begin{equation}
  \bar\nu(M, g)=-48\;.
\end{equation}
Again, reversing the orientation gives a diffeomorphic $G_2$-manifold
with $\bar\nu$-invariant~$48$.

It is calculated in~\cite[Example 8.17]{xtcs} that the $G_2$-manifold~$M$ is
$2$-connected, ${H^4(M) \cong\Z^{109}}$, and~$\div p_1(M)=4$.
\end{ex}

\begin{proof}[Proof of Theorem \ref{Thm4.2}]
According to \cite[Row $b = 86$ of Table 3]{CHNP},
there are rectangular twisted connected sums with the same invariants,
but those have~$\bar\nu=0$ by Corollary~\ref{MainCor}.
Applying Theorem~\ref{thm:torfreeclass}, there is a diffeomorphism between
those rectangular twisted connected sums and Example \ref{ex:pi6rk2}.
Moreover, because they both have $\nu = 24 \in \Z/48$ by~\eqref{eq:relation},
Proposition~\ref{prop:nu_classifies} implies that the diffeomorphism may
be chosen so that the \gtstr s are homotopic.
Nevertheless, since the values of $\bar\nu$ 
differ, the metrics lie in different components of the $G_2$ moduli space.
In other words, while the equality of $\nu$ has helped ensure that the two
torsion-free \gtstr s can be connected by a path of \gtstr s,
the differing values of $\bar \nu$ mean that they cannot possibly
be connected by a path of torsion-free \gtstr s. 
\end{proof}

\section{The Signature Eta Invariant}\label{Kap6}
We use the gluing formulas of Bunke and Kirk and Lesch,
in particular Theorem~8.12
and~(8.32) of~\cite{KiLe}, to determine the $\eta$-invariant
of the odd signature operator~$B_M$ on~$M_\ell$.
We show that the two halves~$M_\pm$ do not contribute if~$k_\pm\le 2$,
and we compute the gluing contribution in terms of the gluing
angle~$\thet$ and the integer~$m_\rho(\latc;N_+,N_-)\in\Z$
determined by the configuration~$N_+$,
$N_-\subset \latc$, see Definition~\ref{def2.6}.
Throughout this section, $H^\bullet(\punkt)$ will always denote
cohomology with real coefficients.

\subsection{The Gluing Formula for the Odd Signature Operator}\label{Sect4.1a}
We review the gluing formula by Kirk and Lesch, see Theorem~\ref{Thm2.5} below,
deviating slightly from~\cite{KiLe}.
First, we only consider manifolds of dimension~$4k-1$.
Next, we use a different description~\eqref{eqn:KLidf}
of differential forms near the gluing hypersurface.
Then the boundary operator~$A$ in~\eqref{2.10}
below will be the Hodge-Dirac operator on the gluing hypersurface~$X$.
Finally, we will check that the gluing
contribution~$m_{H^\bullet(\X)}(\lagr_{B_+},\lagr_{B_-})$ can be computed
within the cohomology in middle degree~$2k-1$, see Lemma~\ref{MiddleLem}.

For the moment,
let~$(M,g)$ be a closed oriented Riemannian manifold of dimension~$4k-1$.
We assume that~$M=M_+\cup_XM_-$ with~$X=\del M_+=\del M_-$,
and that~$M$ contains a subset which is isometric to~$(-1,1)\times X$,
with~$X$ itself being identified with~$\{0\}\times X$.

We consider the Hodge operators~$*_M$ and~$*_X$ of~$M$ and~$X$.
For each~$p$, we define
\begin{equation}\label{2.8}
  \bar*_M=*_M\circ(-1)^{k+\frac{p(p+1)}2}\qquad\text{and}\qquad
  \bar*_X=*_X\circ(-1)^{k+\frac{p(p-1)}2}
\end{equation}
on $p$-forms.
Then~$\bar*_M^2=1$ and~$\bar *_X^2=-1$.
The odd signature operator acts
as~$B_{M}=d_M\bar *_M+\bar *_Md_M$ on~$\Omega^\ev(M)$,
see~\mbox{\cite[eq~(4.6)]{APS1}}.
For~$w\in TM$,
let~$\eps_w$ denote exterior multiplication with the 1-form~$g^{TM}(w,\punkt)$,
and let~$\iota_w$ denote its adjoint.
Then~$B_{M}$ is the Dirac operator associated
with the Clifford multiplication
on~$\Lambda^\ev T^*M$ defined for each~$w\in TM$ by
\begin{equation}\label{2.9}
  c^M_w
  =\eps_w\circ\bar*_M+\bar*_M\circ\eps_w
  =(\eps_w-\iota_w)\circ\bar*_M\;.
\end{equation}

Let~$t$ denote a coordinate in the cylinder direction
that is increasing from~$M_-$ towards~$M_+$.
Let~$p_X\colon (-1,1)\times X\to X$ be the projection.
Then there is an identification of bundles
\begin{equation}\label{eqn:KLidf}
  \begin{gathered}
    \Lambda^\ev T^*\bigl((-1,1)\times X\bigr)\cong p_X^*\Lambda^\bullet T^*X\\
    \alpha+\beta\,dt\longmapsto %
    \alpha+\beta
  \end{gathered}
\end{equation}
for all~$\alpha\in\Lambda^\ev T^*X$, $\beta\in\Lambda^\odd T^*X$;
here we have replaced~$*_M$ as in~\cite[Section~8.1]{KiLe}
by~$\eps_{\del_t}+\iota_{\del_t}$
for simplicity---the resulting operators are conjugate
by an isomorphism that acts as the identity on even forms and as~$-*_X$
on odd forms.
Under the identification~\eqref{eqn:KLidf}, the odd signature operator becomes
\begin{equation}
  \begin{gathered}\label{2.10}
    B_M|_{(-1,1)\times X}=\gamma\biggl(\frac\del{\del t}+A\biggr)\;,\\
    \text{with}\qquad
    \gamma=-\bar *_X
    \qquad\text{and}\qquad
    A=d_X+\bar *_Xd_X\bar *_X=d+d^*\;.
  \end{gathered}
\end{equation}
\newif\ifBdryOp\BdryOpfalse
\ifBdryOp
To verify this claim,
consider~$\alpha\in C^\infty((-1,1),\Omega^\ev(X))$
and~$\beta\in C^\infty((-1,1),\Omega^\odd(X))$.
We can check that~$\bar *_M=\bar *_X\circ(\eps_{\del_t}-\iota_{\del t})$.
Hence
\begin{align*}
  B_M(\alpha+\beta\,dt)
  &=\bar *_M\biggl(d_x\alpha
  +\Bigl(\frac{\del\alpha}{\del t}+d_X\beta\Bigr)\,dt\biggr)
  +d_M\bigl(\bar *_X\alpha\,dt+\bar *_X\beta\bigr)\\
  &=-\bar *_X\biggl(d_X\alpha\,dt+\frac{\del\alpha}{\del x}+d_X\beta\biggr)
  +d_X\bar *_X\alpha\,dt+d_X\bar *_X\beta-\bar *_X\,\frac{\del\beta}{\del t}\,dt\\
  &=-\bar *_X\biggl(\frac\del{\del t}(\alpha+\beta\,dt)
  +(d_X+\bar *_Xd_X\bar*_X)(\alpha\,dt+\beta)\biggr)\;.
\end{align*}
\fi

The endomorphism~$\gamma$ induces a complex structure on~$\Omega^\bullet(X)$,
and because it anticommutes with~$A$, also on~$\ker A\cong H^\bullet(X)$.
Together with the $L^2$-Hermitian metric~$g_{L^2}$,
it induces symplectic structures~$\omega$ on both vector spaces.
By the definition of the Hodge operator~$*_X$,
for harmonic forms~$\alpha\in H^p(X)$, $\beta\in H^{4k-2-p}(X)$, we have
\begin{equation}\label{eqn:sympl}
  \omega(\alpha,\beta)=g_{L^2}(\alpha,\gamma\beta)
  =\int_X\alpha\wedge *_X(-\bar *_X\beta)
  =(-1)^{k+\frac{p(p-1)}2-1}\int_X\alpha\wedge\beta\;.
\end{equation}
We also notice that the middle dimensional cohomology~$H^{2k-1}(X)$
is a symplectic subspace
and that~$\omega|_{H^{2k-1}(X)}$ agrees with the restriction
of the symplectic form in~\cite[(8.5)]{KiLe}.

For the gluing formula,
we need boundary conditions for the odd signature operators on~$M_\pm$.
These should lead to self-adjoint operators~$B_{M_\pm}$,
and their $\eta$-invariants should be well-defined.
From the large supply of possible choices
described in~\cite[Definition~2.1]{KiLe},
we will only consider modified APS boundary conditions,
by which we mean the following.
Given a Lagrangian subspace~$\lagr_+\subset\ker A$,
we take as a domain for~$B_+$ the subspace of~$\Omega^\ev(M_+)$
consisting of forms~$\alpha+\beta\,dt$ such that~$\alpha|_X$ and~$\beta|_X$
are perpendicular to~$\lagr_+$ and to all eigenspaces of~$A$
for positive eigenvalues.
Then we let~$\eta_{\APS}(B_{M_+};\lagr_+)$ denote the $\eta$-invariant of~$B_{M_+}$
with respect to these boundary conditions.

For~$B_{M_-}$ we proceed similarly.
However, since~$t$ is an outward normal coordinate for~$M_-$,
the corresponding boundary operator is now given by~$-A$.
Hence, given a Lagrangian subspace~$\lagr_-\subset\ker(A)$,
we now demand that~$\alpha|_X$ and~$\beta|_X$
are perpendicular to~$\lagr_-$ and to all eigenspaces for negative
eigenvalues of~$A$.
We write~$\eta_{\APS}(B_{M_-};\lagr_-)$ for the corresponding $\eta$-invariant.
Then~$\eta_{\APS}(B_{M_\pm};\lagr_\pm)$ corresponds
to~$\eta(B_M,M_\pm;V_\pm\oplus F_0^\pm)$
in the notation of \mbox{\cite[Theorem~8.8]{KiLe}}, with~$V_\pm=\lagr_{B_\pm}$.

\newif\ifPoincareLefschetz\PoincareLefschetzfalse
\ifPoincareLefschetz
To fix natural Lagrangian subspaces~$\lagr_\pm\subset\ker A\cong H^\bullet(\X)$,
let us consider the diagram
\begin{equation*}
  \begin{tikzpicture}
    \matrix (D) [matrix of math nodes, column sep=1cm, row sep=1cm]{
      \cdots & H^p(M_\pm) & H^p(\X) & H^{p+1}(M_\pm,X) & \cdots \\
      \cdots & H^{4k-1-p}(M_\pm,\X)^* & H^{4k-2-p}(\X)^* &
      H^{4k-2-p}(M_\pm)^* & \cdots \\
    } ;
    \begin{scope}[->]
      \draw (D-1-1) -- (D-1-2) ;
      \draw (D-1-2) -- node[above] {$\iota^*$} (D-1-3) ;
      \draw (D-1-3) -- node[above] {$\delta$} (D-1-4) ;
      \draw (D-1-4) -- (D-1-5) ;
      \draw (D-2-1) -- (D-2-2) ;
      \draw (D-2-2) -- (D-2-3) ;
      \draw (D-2-3) -- (D-2-4) ;
      \draw (D-2-4) -- (D-2-5) ;
      \draw (D-1-2) -- node[right] {$\cong$} (D-2-2) ;
      \draw (D-1-3) -- node[right] {$\cong$} (D-2-3) ;
      \draw (D-1-4) -- node[right] {$\cong$} (D-2-4) ;
    \end{scope}
  \end{tikzpicture}
\end{equation*}
The upper row is the long exact de Rham cohomology sequence
of the pair~$(M_\pm,\X)$.
The connecting homomorphism~$\delta$ sends~$[\beta]\in H^p(\X)$
to~$[d\bar\beta]\in H^{p+1}(M_\pm,\X)$,
where~$\bar\beta\in\Omega^p(M_\pm)$ is an arbitrary
smooth extension of~$\beta$.
The lower row is the dualised sequence.
Each vertical arrow maps a cohomology class~$[\alpha]$ to the map
\begin{equation*}
  [\beta]\longmapsto\int\alpha\wedge\beta\;,
\end{equation*}
where the integral is over~$M_\pm$ or~$X$, respectively.
By Poincar\'e-Lefschetz duality, these maps are isomorphisms.
The diagram commutes up to sign.
For the two squares drawn above, this follows from Stokes theorem.

We now choose subspaces
\begin{equation}\label{2.11}
  \lagr_{B_\pm}=\im\bigl(H^\bullet(M_{\pm,\ell})\to H^\bullet(\X)\bigr)
  \subset H^\bullet(\X)\;.
\end{equation}
By a diagram chase, the~$\lagr_{B_\pm}$ are Lagrangian subspaces.
Indeed, if~$\alpha\in H^p(M_\pm)$ and~$\beta\in H^{4k-1-p}(M_\pm)$,
Stokes theorem gives
\begin{equation*}
  (-1)^{k+\frac{p(p-1)}2-1}\omega(\iota^*\alpha,\iota^*\beta)
  =\int_X\iota^*(\alpha\wedge\beta)=\int_{M_\pm}d(\alpha\wedge\beta)=0\;.
\end{equation*}
On the other hand, if~$\gamma\in H^p(M_\pm)$
satisfies~$\omega(\gamma,\iota^*\beta)=0$
for all~$\beta\in H^{4k-1-p}(M_\pm)$,
then~$\gamma$ maps to~$0\in H^{4k-2-p}(M_\pm)^*$,
hence it has a preimage~$\alpha\in H^p(M_\pm)$.
\else %
We now choose subspaces
\begin{equation}\label{2.11}
  \lagr_{B_\pm}=\im\bigl(H^\bullet(M_{\pm,\ell})\to H^\bullet(\X)\bigr)
  \subset H^\bullet(\X)\;.
\end{equation}
By a standard argument involving the long exact cohomology sequences
of the pairs~$(M_\pm,X)$ and Poincar\'e-Lefschetz duality for manifolds
with boundary,
one checks that~$\lagr_{B_\pm}\subset H^\bullet(X)$ are Lagrangian subspaces
with respect to the symplectic structure~\eqref{eqn:sympl}.
\fi

\begin{dfn}[{%
      \cite[Definition~8.14]{KiLe}}]\label{BunkeDef}
  Let~$(H,\gamma,g)$ be a Hermitian vector space,
  where~$\gamma$ denotes the complex structure,
  and let~$\lagr_\pm\subset H$ be two Lagrangian subspaces
  with respect to~$\omega=g(\punkt,\gamma\punkt)$.
  Let~$E_{\pm i}\subset H\otimes_\R\C$ denote the $\pm i$-eigenspace of~$\gamma$.
  Define unitary maps~$\Phi(\lagr_\pm)\colon E_i\to E_{-i}$
  such that~$\lagr_\pm\otimes_\R\C$ is the graph of~$\Phi(\lagr_\pm)$.
  Then the \emph{Maslov angle} between~$\lagr_+$ and~$\lagr_-$ is given as
  \begin{equation*}
    m_H(\lagr_+,\lagr_-)
    =\dim(\lagr_+\cap \lagr_-)
    -\frac1{\pi i}\tr\log\bigl(-\Phi(\lagr_+)\Phi(\lagr_-)^*\bigr)\in\R\;.
  \end{equation*}
\end{dfn}

Here, the branch of the logarithm is chosen such that~$\log(-1)=\pi i$,
which implies that the eigenvalue~$1$ of~$\Phi(\lagr_+)\Phi(\lagr_-)^*$
with multiplicity~$\dim(\lagr_+\cap \lagr_-)$ does not contribute.
We have chosen the name ``Maslov angle'' because of its relation to the Maslov
triple index, \mbox{see \cite[Proposition~1.7]{BuGlu}.}

\begin{rmk}\label{BunkeRem}
Let~$A_\pm$ be the $\R$-linear isometric involutions of~$H$
that anticommute with~$\gamma$
and fix the Lagrangian subspaces~$\lagr_\pm\subset H$ pointwise.
Then
the maps~$\Phi(\lagr_\pm)\colon E_i\to E_{-i}$ and~$\Phi(\lagr_+)\Phi(\lagr_-)^*$
above are given by
\begin{align*}
  \Phi(\lagr_\pm)
  &=2\,\frac{1+i\gamma}2\,\frac{1+A_\pm}2\,\frac{1-i\gamma}2
  =\frac{1+i\gamma}2\,A_\pm\;,\\
  \text{and}\qquad
  \Phi(\lagr_+)\Phi(\lagr_-)^*
  &=\frac{1+i\gamma}2\,A_+A_-\,\frac{1+i\gamma}2
  =\frac{1+i\gamma}2\,A_+A_-\;.
\end{align*}
We thus recover Lesch-Wojciechowski's description in~\cite[Theorem~2.1]{LW},
see also Bunke~\cite[Definition 1.3]{BuGlu}.
Let \mbox{$e^{i\phy_1}$, \dots, $e^{i\phy_k}$} denote the eigenvalues
of~$-A_+A_-|_{E_{-i}}$ with~$\phy_j\in(-\pi,\pi]$, then
\begin{equation}\label{Mdef}
  m_H(\lagr_+,\lagr_-)=-\sum_{\phy_j\ne\pi}\frac{\phy_j}\pi\;.
\end{equation}
\end{rmk}

We come back to the choice of Lagrangian subspaces in~\eqref{2.11}.
To prove that the Maslov angle~$m_{H^\bullet(\X)}(\lagr_{B_+},\lagr_{B_-})$
only depends on the subspaces~$\lagr^3_{B_\pm}=\lagr_{B_\pm}\cap H^3(\X)$,
we write
\begin{equation*}
  \lagr_{B_\pm}=\bigoplus_{p=0}^{4k-2}\lagr_{B_\pm}^p\qquad\text{with}\qquad
  \lagr_{B_\pm}^p=\im\bigl(H^p(M_{\pm,\ell})\to H^p(\X)\bigr)\;.
\end{equation*}
Then~$\lagr_{B_\pm}^{4k-2-p}=(\lagr_{B_\pm}^p)^\perp$
with respect to the intersection form,
so~$\lagr_{B_\pm}^{4k-2-p}=(\gamma \lagr_{B_\pm}^k)^\perp$ with respect to the $L^2$-metric.
In particular,
the involutions~$A_{B_\pm}$ in Remark~\ref{BunkeRem}
preserve each individual~$H^k(X)$.

\begin{lem}\label{MiddleLem}
  We have
	$$m_{H^\bullet(\X)}(\lagr_{B_+},\lagr_{B_-})
	=m_{H^{2k-1}(\X)}\bigl(\lagr^{2k-1}_{B_+},\lagr^{2k-1}_{B_-}\bigr)\;.$$
\end{lem}

\begin{proof}
  We note that~$-A_{B_+}A_{B_-}$ commutes with~$\gamma=-\bar*_\X$,
  which interchanges~$H^p(\X;\bbc)$ and~$H^{4k-2-p}(\X;\bbc)$.
  The action of~$-A_{B_+}A_{B_-}$
  on~$E_{-i}\cap\bigl(H^p(\X;\bbc)\oplus H^{4k-2-p}(\X;\bbc)\bigr)$
  is therefore iso\-morphic to the complexification of the action
  on the real vector space~$H^p(\X)$ for all~$p<2k-1$.
  Because of this, the spectrum
  of~$-\Phi(\lagr_{B_+})\Phi(\lagr_{B_-})^*|_{H^p(\X)\oplus H^{4k-2-p}(\X)}$
  is invariant under complex conjugation for all~$p<2k-1$.
  Hence for these~$p$, we have
  \begin{equation*}
    m_{H^p(\X)\oplus H^{4k-2-p}(\X)}
    \bigl(\lagr^p_{B_+}\oplus \lagr^{4k-2-p}_{B_+},\lagr^p_{B_-}\oplus \lagr^{4k-2-p}_{B_-}\bigr)
    =0\;,
  \end{equation*}
  and the claim follows.
\end{proof}

We can now state the slightly simplified gluing formula.

\begin{thm}[{\cite[equation (8.32)]{KiLe}}]\label{Thm2.5}
  Let~$M=M_+\cup_\X M_-$, and let~$B_M$
  denote the odd signature operator.
  Then
  \begin{equation*}
    \eta(B_M)
    =\eta_{\APS}(B_{M_-};\lagr_{B_-})+\eta_{\APS}(B_{M_+};\lagr_{B_+})
	+m_{H^{2k-1}(X)}\bigl(\lagr^{2k-1}_{B_+},\lagr^{2k-1}_{B_-}\bigr)\;.
  \end{equation*}
\end{thm}

\subsection{A family of gluing metrics}\label{Sect4.1}
Recall that we constructed a closed, but not torsion-free
$G_2$-structure~$\phy_\ell$ %
in Section~\ref{subsec:glue}, equations~\eqref{2.4}--\eqref{2.6a}.
By Theorem~\ref{Thm2.1}, for each~$\ell\gg1$,
there exists a torsion free $G_2$-structure~$\bar\phy_\ell$
near~$\phy_\ell$ in the same cohomology class.
Let~$\bar g_\ell$ denote the Riemannian metric induced by~$\bar\phy_\ell$.
To apply the gluing formula above,
we need to equip the manifold~$M_\ell$ in~\eqref{eqn:Ml}
with a new Riemannian metric~$g_\ell$ that is of product type
near the gluing hypersurface~$\X$.
Although we are interested in~$(M_\ell,\bar g_\ell,\bar\phy_\ell)$,
most computations will be done on~$(M_\ell,g_\ell,\phy_\ell)$.
The metric~$g_\ell$ will be sufficiently close to the metric
associated with~$\phy_\ell$.
Theorem~\ref{Thm2.1} can then be used to compare connections, Dirac operators,
curvature tensors and related geometric magnitudes defined
by the metrics~$g_\ell$ and~$\bar g_\ell$.

We still use the coordinates~$t_\pm$ and~$t$ of Section~\ref{Kap2}
with~$t_\pm-\ell=2\mp t$, see~\eqref{tCoord}.
Then~$M_\ell$ contains a cylindrical
piece diffeomorphic to~$X\times(-\ell-2,\ell+2)$,
with~$X$ as in~\eqref{eqn:X}.
Let~$t\colon M_\ell\to\R$ be a smooth function that agrees with the cylindrical
coordinate on this region and takes values
outside~$(-\ell-2,\ell+2)$ otherwise.
We may rewrite~\eqref{eqn:Ml} as
\begin{equation*}
  M_\ell=M_{-,\ell+1}\cup_{[-1,1]\times X}M_{+,\ell+1}\;.
\end{equation*}

Let~$g^{V_\pm}$ denote the asymptotically cylindrical Calabi-Yau metric
on~$V_\pm$, which is $\Gamma_\pm$-invariant.
Let~$g^{\Sigma_\pm}+du_\pm^2+dt_\pm^2$ denote the
cylindrical metric to which~$g^{V_\pm}$ is asymptotic.
Let~$\rho\colon\R\to[0,1]$ be a smooth cutoff function
as in~\eqref{2.6},
so~$\rho(x)=0$ for~$x\le 0$ and~$\rho(x)=1$ for~$x\ge 1$.
We define a $\Gamma_\pm$-invariant metric~$g^{V_\pm}_{\ell+1}$
on~$V_{\pm,\ell+1}=V_\pm\setminus((\ell+3,\infty)\times S^1_{\lnn_\pm}\times\K_\pm)$ by
\begin{equation}\label{eq:gldef}
  g^{V_\pm}_\ell
  =\bigl(1-\rho(t_\pm-\ell)\bigr)\,g^{V_\pm}
  +\rho(t_\pm-\ell)\,\bigl(g^{\Sigma_\pm}+du_\pm^2+dt_\pm^2\bigr)\;,
\end{equation}
take the product with a circle~$S^1_{\xi_\pm}$ of length~$\xi_\pm$,
and finally obtain a metric~$g^{M_\pm}_{\ell+1}$
on~$M_{\pm,\ell+1}=(V_{\pm,\ell+1}\times S^1_{\xi_\pm})/\Gamma_\pm$.
If we regard~$M_{\pm,\ell+1}$ as subsets of~$M_\ell$,
then the metrics~$g^{M_\pm}_{\ell+1}$ agree on the
intersection~$M_{+,\ell+1}\cap M_{-,\ell+1}\cong[-1,1]\times X$.
Hence we obtain a Riemannian metric~$g_\ell$ on~$M_\ell$.

\begin{rmk}\label{Rem2.2}
  There exists a constant~$c>0$ such that the
  Riemannian metric~$g_\ell$ has the following properties,
  which will be needed below.
  \begin{enumerate}
  \item\label{2.g.1}
    The subset~$M_{\pm,\ell+1}\subset(M_\ell,g_\ell)$
    corresponding to~$\pm t\ge-1$ is isometric to a twisted
    product~${(V_{\pm,\ell+1}\times S^1_{\xi_\pm})/\Gamma_\pm}$
    of the space~$(V_{\pm,\ell+1},g^{V_\pm}_\ell)$ and a circle~$S^1_{\xi_\pm}$
    of length~$\xi_\pm$.
  \item\label{2.g.2}
    On the smaller subset~$M_{\pm,\ell-2}$, that is,
    for~$\pm t\ge 2$, the factor~$\bigl(V_{\pm,\ell-2},g^{V_\pm}_\ell\bigr)$
    in~\ref{2.g.1} is isometric
    to a subset of the original asymptotically cylindrical
    Calabi-Yau manifold~$\bigl(V_\pm,g^{V_\pm}\bigr)$.
  \item\label{2.g.3}
    The subset~$X\times[-1,1]=M_{-\ell+1}\cap M_{+,\ell+1}$
    is the Riemannian product of the K3 surface~$\Sigma$,
    the torus~$T^2\cong(S^1_{\zeta_\pm}\times S^1_{\xi_\pm})/\Gamma_\pm$
    and the interval~$[-1,1]$.
  \item\label{2.g.4} For~$1\le\pm t\le 2$ and for all~$k$
    we have~$\norm{g_\ell|_{X\times(\pm[1,2])}-g^X\oplus dt^2}_{C^k}=O(e^{-c\ell})$.
  \item\label{2.g.5}
    Let~$\bar g_\ell$ denote the $G_2$-metric induced by~$\bar\phy_\ell$.
    For each~$k$,
    there is an estimate of the form
    \begin{equation*}
      \norm{g_\ell-\bar g_\ell}_{C^k}=O\bigl(e^{-c\ell}\bigr)\;.
    \end{equation*}
  \end{enumerate}
  It follows from~\ref{2.g.2} and~\ref{2.g.3} that~$g_\ell$ is induced
  by~$\phy_\ell$ and the local holonomy group of~$g_\ell$ is a subgroup of~$G_2$
  except over the set~$X\times([-2,-1]\cup[1,2])$,
  where the metric is controlled by~\ref{2.g.1} and~\ref{2.g.4}.
\end{rmk}

\subsection{Spectral Symmetry of the Odd Signatures Operators on the Halves}
\label{Abs2.3}
We recall the action of~$\Gamma_\pm$
on~$\widetilde M_\pm=V_\pm\times S^1_{\xi_\pm}$ described
at the end of Section~\ref{subsec:acyl}.
The following proposition is the central result of this section.

\begin{prop}\label{prop2.3}
  Assume that~$\Gamma_\pm\cong\Z/k_\pm$ with~$k_\pm\in\{1,2\}$.
  Then
  \begin{equation*}
    \eta_{\APS}\bigl(B_{M_\pm,\ell};\lagr_{B_\pm}\bigr)=0\;.
  \end{equation*}
\end{prop}

In the sequel~\cite{GN},
we will consider~$\Gamma_\pm\cong\Z/k_\pm$ where~$k_\pm$ may be larger than~$2$.
In this case, the $\eta$-invariant~$\eta_{\APS}\bigl(B_{M_\pm};\lagr_{B_\pm}\bigr)$
can be nonzero.

\begin{proof}
  For simplicity, we restrict our attention to~$M_{-,\ell}$.
  We consider an involutive isometry~$\tilde\kappa_-$
  of~$\widetilde M_-$
  that acts as identity on~$V_-$ and as a reflection on~$S^1_{\xi_-}$.
  To see that~$\tilde\kappa_-$ descends to an isometry on~$M_-$,
  assume that~$\Gamma_-$ is nontrivial.
  Then~$k_-=2$ by assumption, and there exists exactly one
  nontrivial element~$\tau_-\in\Gamma_-$,
  which rotates~$S^1_{\xi_-}$ by~$\pi$.
  We see that~$\tilde\kappa_-$ and~$\tau_-$ commute,
  so~$\tilde\kappa_-$ indeed descends to an orientation
  reversing isometry~$\kappa_-$ of~$M_-$.

  Let~$\anglex$ denote the coordinate of the exterior circle~$S^1_{\xi_-}$.
  Let~$c_\anglex$ denote Clifford multiplication by~$\del_\anglex$
  as in~\eqref{2.9}.
  Then~$c_\anglex$ anticommutes with~$B_{M_-}|_{V_-}$, and
  the reflection of~$S^1_{\xi_-}$ anticommutes with~$\frac\del{\del \anglex}$.
  We lift~$\kappa_-$ to an action~$\bar \kappa_-$ on~$\Omega^\ev(M_-)$
  by putting
  \begin{equation*}
    \bar \kappa_-\alpha=c_\anglex\kappa_-^*\alpha
    \qquad\text{for all }\alpha\in\Omega^{2p}(M)\;.
  \end{equation*}
  Because~$B_{M_-}$ is a Dirac operator, it anticommutes with~$\bar \kappa_-$.
  The same holds for~$B_{M_-,\ell}$ on~$M_{-,\ell}$
  because the metric~$g_\ell$
  is still a product metric on~$\widetilde M_{-,\ell}$
  by Remark~\ref{Rem2.2}~\ref{2.g.1}.
  
  On the other hand, $\bar \kappa_-$ also anticommutes with~$\gamma=-\bar *_X$
  by~\eqref{2.10},
  and hence, $\bar \kappa_-$ commutes with~$A=\bar *_XB_{M_-,\ell}|_X$
  and preserves the unmodified APS-boundary conditions.
  Because~$\bar\kappa_-$ acts on~$\ker(B_{M_-,\ell})$ and on~$\ker(A)$,
  it also preserves the Lagrangian~$\lagr_{B_-}$ of~\eqref{2.11},
  and hence the domain of the operator~$B_{M_-,\ell}$ under the modified
  APS boundary conditions introduced in Section~\ref{Sect4.1a}.
  It follows that with these boundary conditions,
  the operator~$B_{M_-,\ell}$ has symmetric spectrum,
  and so the $\eta$-invariant~$\eta_{\APS}(B_{M_-,\ell};\lagr_{B_-})$ vanishes.
\end{proof}

\subsection{The Maslov Angle of the Odd Signature Operator}\label{Abs2.5}
We now investigate the \linebreak number $m_{H^3(X)}(\lagr^3_{B_+},\lagr^3_{B_-})$.
Let~$\widetilde X_\pm=\del\widetilde M_{\pm,\ell}
=\K_\pm\times S^1_{\lnn_\pm}\times S^1_{\lnx_\pm}$.
We remark that the action of each element of~$\Gamma_\pm\cong\Z/k_\pm$
on~${\widetilde X_\pm}$
is homotopic to the identity because~$\K$ is fixed pointwise
and the two circles~$S^1_{\lnn_\pm}$ and~$S^1_{\lnx_\pm}$ are rotated
by elements of~$\Gamma_\pm$.
Hence $\Gamma_\pm$ acts trivially on~$H^\bullet(\widetilde \X_\pm)$.
independent of~$k_\pm$.
In particular, pullback along~$\tilde X_\pm\to\tilde X_\pm/\Gamma_\pm\cong\X$
induces isomorphisms~$H^\bullet(\X)\cong H^\bullet(\widetilde \X_\pm)$.
Because ${H^1(\K)=H^3(\K)=0}$,
the space~$H^3(\X)$ takes the form
\begin{equation*}
  H^3(\X)\cong H^3(\widetilde \X_\pm)\cong H^2(\K)\otimes_\bbr H^1(T^2)\;,
\end{equation*}
and the intersection form on~$H^3(\X)$ is the tensor product of the intersection
forms on~$H^2(\K)$ and on~$H^1(T^2)$.
The same holds for the $L^2$-metrics and the Hodge star operators.
Note that by Definition~\ref{BunkeDef} the Maslov
angle~$m_{H^3(X)}(\lagr_{B_+}^3,\lagr_{B_-}^3)$ depends
on~$\gamma=-\bar{\mathord{*}}_X$ and on the $L^2$-metric.

Because $\Gamma_\pm$ acts trivially on $\K$ and we use cohomology with real
coefficients, %
it is clear that
\begin{equation}\label{HodgeLagrangian}
	\lagr_{B_\pm}^3=\im\bigl(H^3(V_\pm {\times} S^1_{\lnx_\pm})^{\Gamma_\pm}\to H^3(\X)\bigr)
        \cong\im\bigl(H^3(V_\pm {\times} S^1_{\lnx_\pm}) \to H^3(\widetilde \X_\pm)\bigr) \;.%
\end{equation}
Let~$d\anglen_\pm\in\Omega^1(S^1_{\lnn_\pm})$ and~$d\anglex_\pm\in\Omega^1(S^1_{\lnx_\pm})$
be generators of~$H^2(T^2)$.
Recall the polarising lattices $N_\pm$ of Definition~\ref{NpmDef}
and write~$N_{\pm,\R}=N_\pm\otimes_\Z\R=\im\bigl(H^2(V_\pm) \to H^2(\K)\bigr)$.
From the K\"unneth formula, we see that there is another subspace~$T_{\pm,\R}\subset H^2(\K)$
such that~$\lagr_{B_\pm}$ is the direct sum of the two subspaces
\begin{align*}
  N_{\pm,\R}\,d\anglex_\pm&=\im\bigl(H^2(V_\pm)\otimes H^1(S^1_{\lnx_\pm})
  \longrightarrow H^2(\K)\otimes H^1(S^1_{\lnx_\pm})\bigr)\;,\\
  T_{\pm,\R}\,d\anglen_\pm&=\im\bigl(H^3(V_\pm)
  \longrightarrow H^2(\K)\otimes H^1(S^1_{\lnn_\pm})\bigr)\;.
\end{align*}
Because we know that~$\lagr_{B_\pm}^3\subset H^3(\X)$ is a Lagrangian subspace
with respect to the intersection form,
we immediately see that~$T_{\pm,\R}=N_{\pm,\R}^\perp$ with respect
to the intersection form.

To describe the Maslov angle as in Definition~\ref{BunkeDef},
we need~$T_{\pm,\R}=N_{\pm,\R}^\perp$ also with respect to the $L^2$-metric.
The space~$N_\pm$ clearly contains the K\"ahler form of~$\Sigma_\pm$.
By~\eqref{2.2} and~\eqref{2.6} and because~$dt=dt_-$,
the limiting value of the holomorphic volume form~$\Omega_{V_\pm,\infty}$
on~$V_\pm$ is given as
\begin{equation*}
  \Omega_{V_\pm,\ell}|_{(-1,1)\times X}
  =\bigl(du_\pm\pm i\,dt\bigr)\wedge\Omega_{\Sigma_\pm}\;,
\end{equation*}
where~$\Omega_{\Sigma_\pm}$ denotes the holomorphic volume form of~$\Sigma_\pm$.
Therefore, $\Re\Omega_{\K_\pm}$, $\Im\Omega_{\K_\pm}\in T_\pm$.
Because~$(\omega_{\Sigma_\pm},\re\Omega_{\Sigma_\pm},\im\Omega_{\Sigma_\pm})$
forms a basis of~$H^{2,+}(\K)$,
we obtain the first equality in
\begin{align*}
  H^{2,+}(\K)
  &=\bigl(H^{2,+}(\K)\cap N_{\pm,\R}\bigr)\oplus\bigl(H^{2,+}(\K)\cap T_{\pm,\R}\bigr)\;,\\
  H^{2,-}(\K)
  &=\bigl(H^{2,-}(\K)\cap N_{\pm,\R}\bigr)\oplus\bigl(H^{2,-}(\K)\cap T_{\pm,\R}\bigr)\;.
\end{align*}
The second follows because~$T_{\pm,\R}$ is the orthogonal complement
of~$N_{\pm,\R}$ with respect to the intersection form.
This implies in particular that~$N_\pm$ is perpendicular to~$T_\pm$
also with respect to the $L^2$-metric.
We can now construct the reflections~$A_\pm$ about~$N_\pm$
and conclude that the condition \eqref{eq:preserve} in Definition~\ref{def:angles} is satisfied.

Let~$A^3_{B_\pm}$ denote the reflections of~$H^3(X)$
along the Lagrangians~$\lagr_{B_\pm}^3$,
and let~$A_{\anglex_\pm}$ denote the reflections %
of~$H^1(T^2)$ at~$[d\anglex_\pm]\R$, respectively.
Then
\begin{equation*}
	A^3_{B_\pm}=A_{\pm}\otimes A_{\anglex_\pm}
	\in\Aut\bigl(H^2(\K)\bigr)\otimes\Aut\bigl(H^1(T^2)\bigr)
        =\Aut\bigl(H^3(\X)\bigr)\;.
\end{equation*}

From~\eqref{2.8} and~\eqref{2.10},
we see that
\begin{equation*}
  \gamma|_{H^3(X)}
  =*_X|_{H^3(X)}\cong *_\Sigma|_{H^2(\Sigma)}\otimes *_{T^2}|_{H^1(T^2)}\;.
\end{equation*}
The $\pm1$-eigenspaces of~$*_\K$ on~$H^2(\K)$
are the spaces~$H^{2,\pm}(\K)$ of selfdual and antiselfdual forms.
Hence,
the $(-i)$-eigenspace of~$\gamma$ takes the form
\begin{equation}\label{2.14}
  H^{2,+}(\K)\otimes E_{-i} \; \oplus \; H^{2,-}(\K)\otimes E_i\;,
\end{equation}
where~$E_{\pm i}\subset H^1(T^2;\C)$
denote the corresponding eigenspaces of~$*_{T^2}$.

Now suppose that the angle from~$e_{\anglex_+}$ to~$e_{\anglex_-}$
is~$\thet\in(0,\pi)$,
then~$-A_{\anglex_+}A_{\anglex_-}$ is a rotation by~$\rho=\pi-2\thet$,
so
\begin{equation*}
  -A^3_{B_+}A^3_{B_-}=A_+A_-\otimes e^{\rho\,*_{T^2}}\;,
\end{equation*}
where we regard~$*_{T^2}$ as a complex structure on~$H^1(T^2)$.

Let~$\alpha_1^+$, $\alpha_2^+$, $\alpha_3^+$
and~$\alpha_1^-$, \dots, $\alpha_{19}^-\in(-\pi,\pi]$ denote the angles
of the orthogonal auto\-morphism $A_+A_-$
acting on~$H^{2,\pm}(\K;\C)$ %
as in Definition \ref{def:angles}.
Note that these angles occur in pairs with the same superscript
of the form~$\alpha$, $-\alpha$,
except for the angles~$0$ and~$\pi$, which may occur arbitrarily often.
By~\eqref{2.14}, the corresponding angles of~$-A^3_{B_+}A^3_{B_-}$
on the $(-i)$-eigenspace
of~$\gamma=\mathord{*_\K}\otimes\mathord{*_{T^2}}$ are
\begin{equation}\label{eqn:EffAngles}
  \alpha_1^+-\rho,\dots,\alpha_3^+-\rho,
  \alpha_1^-+\rho,\dots,\alpha_{19}^-+\rho\mod2\pi\;.
\end{equation}

\begin{prop}\label{Mprop}
  The Maslov angle in~\eqref{Mdef} is given by
  \begin{multline}\label{eq:mH3}
    m_{H^3(\X;\bbc)}(\lagr^3_{B_+},\lagr^3_{B_-})
    =-16\frac\rho\pi\\
    \begin{aligned}
      &\qquad
      -\Sign\rho\,\#\bigl\{\,j
      \bigm|\alpha_j^+\in\{\pi-\abs\rho,\pi\}\,\bigr\}
      -2\Sign\rho\,\#\bigl\{\,j
      \bigm|\alpha_j^+\in(\pi-\abs\rho,\pi)\,\bigr\}\\
      &\qquad
      +\Sign\rho\,\#\bigl\{\,j
      \bigm|\alpha_j^-\in\{\pi-\abs\rho,\pi\}\,\bigr\}
      +2\Sign\rho\,\#\bigl\{\,j
      \bigm|\alpha_j^-\in(\pi-\abs\rho,\pi)\,\bigr\}\;,
    \end{aligned}
  \end{multline}
  where by convention, $\Sign0=0$.
\end{prop}

Note that this number depends only on the configuration~$(\latc;N_+,N_-)$
from Definition~\ref{ConfDef}.

\begin{proof}
  Recall that to apply~\eqref{Mdef},
  we have to represent all the angles~$\alpha^\pm_j\mp\rho$ in~$(-\pi,\pi]$,
  and that angles in~$\{0,\pi\}$ do not contribute
  to~$m_{H^3(X;\C)}(\lagr^3_{B_+},\lagr^3_{B_-})$.
  Without these restrictions,
  we would expect that~$m_{H^3(X;\C)}(\lagr^3_{B_+},\lagr^3_{B_-})=-16\,\frac\rho\pi$
  because all angles except~$0$ and~$\pi$ occur in pairs of opposite signs,
  and~$\K$ has signature~$-16$.
  We will now see how individual angles may modify this result.
  For simplicity, we first consider only the angles~$\alpha^-_j$,
  and we assume~$\rho\in[0,\pi)$.
  \begin{enumerate}
  \item\label{case:rho0} The case~$\rho=0$ occurs if~$\thet=\frac\pi2$.
    This happens in particular for the classical twisted connected sums
    of~\cite{Kovalev}.
    The angles of~$-A_{B_+}^3A_{B_-}^3|_{E_{-i}}$ come in pairs,
    except for~$0$ and~$\pi$, which do not contribute.
    Hence,
    \begin{equation*}
      m_{H^3(X;\C)}(\lagr^3_{B_+},\lagr^3_{B_-})=0=-16\,\frac\rho\pi\;.
    \end{equation*}
  \end{enumerate}
  From now on, assume~$\rho\in(0,\pi)$.
  We only consider angles~$\alpha^-_j\in[0,\pi]$.
  If~$\alpha^-_j\in(0,\pi)$,
  there will be another index~$j'$ such that~$\alpha^-_{j'}=-\alpha^-_j$.
  We will treat these two angles together,
  so in the end, all angles will be accounted for.
  \begin{enumerate}\setcounter{enumi}1
  \item\label{case:alpha0}
    A single angle~$\alpha^-_j=0$ becomes~$\rho$ in~\eqref{eqn:EffAngles}
    and contributes by~$-\frac\rho\pi$ to~$m_{H^3(\X;\bbc)}(\lagr^3_{B_+},\lagr^3_{B_-})$
    in accordance with~\eqref{eq:mH3}.
  \item\label{case:alphasmall}
    Assume that~$\alpha^-_j\in(0,\pi-\rho)$.
    Then~$\alpha^-_j+\rho\in(-\pi,\pi)$ and~$-\alpha^-_j+\rho\in(-\pi,\pi)$,
    so both angles together contribute to~$m_{H^3(\X;\bbc)}(\lagr^3_{B_+},\lagr^3_{B_-})$ by
    \begin{equation*}
      -\frac{\alpha_j^-+\rho}\pi-\frac{\alpha_{j'}^-+\rho}\pi=-2\,\frac\rho\pi\;.
    \end{equation*}
  \item\label{case:alpharho}
    Assume~$\alpha^-_j=\pi-\rho$.
    Then~$\alpha^-_j+\rho=\pi$ does not contribute,
    but we also have~$\alpha^-_{j'}+\rho=2\rho-\pi\in(-\pi,\pi)$.
    Both angles together contribute to~$m_{H^3(\X;\bbc)}(\lagr^3_{B_+},\lagr^3_{B_-})$ by
    \begin{equation*}
      -\frac{\alpha^-_{j'}+\rho}\pi=-2\,\frac\rho\pi+1\;.
    \end{equation*}
    This agrees with~\eqref{eq:mH3} because one of the two angles is
    exactly~$\pi-\abs\rho$.
  \item\label{case:alphapi}
    Assume that we have a single angle~$\alpha^-_j=\pi$.
    Then we have to replace~$\alpha^-_j+\rho=\rho+\pi>\pi$
    by~$\rho-\pi\in(-\pi,\pi)$.
    The contribution of this angle to~$m_{H^3(\X;\bbc)}(\lagr^3_{B_+},\lagr^3_{B_-})$ is
    \begin{equation*}
      -\frac{\rho-\pi}\pi=-\frac\rho\pi+1\;.
    \end{equation*}
    Again, this fits with~\eqref{eq:mH3}.
  \item\label{case:alphalarge}
    Finally assume~$\pi-\rho<\alpha^-_j<\pi$.
    Because~$\alpha^-_j+\rho>\pi$,
    we have to replace it by~$\alpha^-_j+\rho-2\pi\in(-\pi,\pi)$.
    On the other hand,
    we have~$\alpha^-_{j'}+\rho=-\alpha^-_j+\rho\in(-\pi,\pi)$.
    These two angles together now contribute
    to~$m_{H^3(\X;\bbc)}(\lagr^3_{B_+},\lagr^3_{B_-})$ by
    \begin{equation*}
      -\frac{\alpha^-_j+\rho-2\pi}\pi-\frac{-\alpha^-_j+\rho}\pi
      =-2\,\frac\rho\pi+2\;.
    \end{equation*}
    This explains why angles in~$(\pi-\abs\rho,\pi)$
    count twice in~\eqref{eq:mH3}.
  \end{enumerate}
  The situation for~$\rho<0$ is analogous.
  Also, the angles~$\alpha^+_j$ behave for~$\rho>0$ as the angles~$\alpha^-_j$
  for~$\rho<0$ and vice versa.
  This completes the proof.
\end{proof}

Recall the integer~$m_\rho(\latc;N_+,N_-)$ defined in  Definition \ref{def2.6}.

\begin{thm}\label{thm2.6}
  Let~$\rho=\pi-2\thet$.
  Then the $\eta$-invariant of the odd signature operator on an extra-twisted
  connected sum~$M_\ell$ with gluing angle $\thet$ is given by
  \begin{gather*}
    \eta(B_M)
    =\eta_{\APS}\bigl(B_{M_{+,\ell}};\lagr_{B_+}\bigr)
    +\eta_{\APS}\bigl(B_{M_{-,\ell}};\lagr_{B_-}\bigr)
    -16\frac\rho\pi+m_\rho(\latc;N_+,N_-)\;.\qed
  \end{gather*}
\end{thm}

\begin{proof}
  We have determined the angles~$\alpha_1^+$, $\alpha_2^+$, $\alpha_3^+$
  in~\eqref{eq2:+rotation}.
  We distinguish two cases.
  \begin{itemize}
  \item If~$\thet=\frac\pi2$, we have~$\rho=0$,
    and there is no contribution to~$m_{H^3(\X;\bbc)}(\lagr_{B_+},\lagr_{B_-})$.
  \item If~$\thet\ne\frac\pi2$,
    then~$\alpha^+_{1,2}\equiv\pm(\pi-\abs\rho)$ and~$\alpha^+_3=0$
    are pairwise different, and
    \begin{equation}\label{eq:PosAngles}\kern1em
      -\Sign\rho\,\#\bigl\{\,j
      \bigm|\alpha_j^+\in\{\pi-\abs\rho,\pi\}\,\bigr\}
      -2\Sign\rho\,\#\bigl\{\,j
      \bigm|\alpha_j^+\in(\pi-\abs\rho,\pi)\,\bigr\}
      =-\Sign\rho\;.
    \end{equation}
  \end{itemize}
  Setting~$\Sign\rho=0$ if~$\rho=0$ unifies both cases.
  
  Using Definition \ref{def2.6}
  and~\eqref{eq:PosAngles}, we rewrite equation~\eqref{eq:mH3}
  in Proposition~\ref{Mprop} as
  \begin{equation*}
    m_{H^3(\X;\bbc)}(\lagr^3_{B_+},\lagr^3_{B_-})
    =-16\frac\rho\pi+m_\rho(\latc;N_+,N_-)\;.
  \end{equation*}
  We insert this into the gluing formula in Theorem~\ref{Thm2.5}
  and obtain our claim.
\end{proof}

\section{The Spinor Eta Invariant}\label{Kap3}
To compute the $\eta$-invariant~$\eta(D_M)$ of the spin Dirac operator,
we proceed analogously.
Both the $G_2$-metric~$\bar g_\ell$ and the Calabi-Yau metric
on the asymptotically cylindrical manifolds~$V_\pm$ are Ricci flat,
and so one expects that harmonic spinors are parallel.
Hence, we might hope that the gluing contribution to the spinorial
$\eta$-invariant has an interpretation in terms of these parallel spinors.

However, some care is needed, because we have
to work with the gluing metric~$g^{TM}_\ell$ on~$M_\ell$
of Section~\ref{Sect4.1}, which is not a $G_2$-metric in general.
Neither are the metrics induced on~$\widetilde M_{\pm,\ell}$ products
of Calabi-Yau metrics with the metric on the exterior circles~$S^1_{\xi_\pm}$.
We therefore also modify the Dirac operator on~$M_\ell$ in such a way that
its kernel is one-dimensional as for a true $G_2$-Dirac operator,
and under suitable boundary conditions, the kernels on~$M_{\pm,\ell}$
are two-dimensional.

We can then give a tractable gluing formula for the $\eta$-invariants
of the modified Dirac operators. And because our construction
rules out spectral flow, we get a good approximation of the
$\eta$-invariant for the $G_2$-metric~$\bar g_\ell$ that we need to
compute~$\bar\nu(M_\ell,\bar g_\ell)$ according to Definition~\ref{MainDef}.

\subsection{A Deformation of the Spin Dirac Operator}\label{Abs3.2}
We will construct a family of metric~$g_{\ell,w}$ on~$M_\ell$
depending on a parameter~$w\in[0,1]$
that interpolates between the $G_2$-metric~$\bar g_\ell$
from Theorem~\ref{Thm2.1} for~$w=0$ and the gluing metric~$g_\ell$
of Section~\ref{Sect4.1} for~$w\in\bigl[\frac 12,1\bigr]$.
We then also modify the Dirac operator for~$w>0$ such that~$D_{\ell,w}$ has
a one-dimensional kernel for all~$w\in[0,1]$,
and such that for~$w=1$, the kernel on~$M_{\pm,\ell+1}$
becomes two-dimensional under weak APS boundary conditions, see below.

For~$w\in\bigl[0,\frac12\bigr]$, we interpolate between~$\bar g_\ell$
and~$g_\ell$ using a smooth cutoff function in the parameter~$w$.
By Remark~\ref{Rem2.2}~\ref{2.g.5}, we conclude that
\begin{equation*}
  \norm{g_{\ell,w}-\bar g_\ell}_{C^k}=O(e^{-c\ell})
  \qquad\text{and}\qquad
  \norm{g_{\ell,w}-g_\ell}_{C^k}=O(e^{-c\ell})
\end{equation*}
uniformly in~$w\in[0,1]$,
where~$c$ still is the constant from Theorem~\ref{Thm2.1}.

If we regard the metrics above as bundle
isomorphisms~$g_{\ell,w}\colon TM\to T^*M$,
then~$\bar g_\ell^{-1}\circ g_{\ell,w}$ is a $\bar g_\ell$-selfadjoint
automorphism close to the identity.
It has a unique positive square-root~$\Phi_{\ell,w}\in\Aut TM$
such that~$\Phi_{\ell,w}^*\bar g_\ell=g_{\ell,w}$
and
\begin{equation*}
  \norm{\id_{TM}-\Phi}_{C^k}=O\bigl(e^{-c\ell}\bigr)
\end{equation*}
for all~$w\in[0,1]$.

Let~$SM_\ell\to M_\ell$ denote the spinor bundle
for the metric~$\bar g_\ell=g_{\ell,0}$.
If~$\bar c\colon TM\otimes SM\to SM$ denotes Clifford multiplication
on~$(M_\ell,\bar g_\ell)$,
then~$c=\bar c\circ(\Phi_{\ell,w}\otimes\id_{SM})$ is a Clifford multiplication
with respect to the metric~$g_{\ell,w}$.
Hence, we can identify the spinor bundles for~$(M,g_{\ell,w})$ with~$SM$
as bundles with metrics for all~$w\in[0,1]$.
Then the parallel spinor~$s_\ell$ on~$(M,\bar g_\ell)$
can be regarded as a unit section
of~$S_{\ell,w}M$ for all~$w\in[0,1]$.

Recall that by Remark~\ref{Rem2.2}~\ref{2.g.2}, \ref{2.g.3},
the $G_2$-structure~$\phy_\ell$ of~\eqref{2.6a}, \eqref{2.7} is parallel
with respect to the gluing metric~$g_\ell$
on the subset~$M\setminus(X\times([-2,-1]\cup[1,2]))\subset M_\ell$.
We want to replace~$s_\ell$ by a unit spinor~$s_{\ell,w}$ that defines~$\phy_\ell$
over this subset.

\begin{rmk}
  We will need an explicit relation between
  a $G_2$-structure~$\phy$ on a spin $7$-manifold~$M$ and
  its defining spinor~$s$.
  \begin{itemize}
  \item For~$x$, $y$, $z\in T_pM$,
    we have~$\phy_p(x,y,z)=\pm\<c_xc_yc_zs(p),s(p)\>$.
  \item Clifford multiplication with~$\phy_p$ acts on~$S_pM$
    such that~$c(\phy_p)(s(p))=\pm 7s(p)$ and~$c(\phy_p)(v)=\mp v$
    if~$v\in s(p)^\perp\subset S_pM$.
  \end{itemize}
  To check this,
  one can use octonion multiplication~$\Im\bbo\times\bbo\to\bbo$ as an
  explicit model for Clifford multiplication of~$T_pM$ on~$S_pM$. 
  The signs in the formulas above are a matter of convention and
  play no role in the sequel.
\end{rmk}

We define~$\phy_{\ell,w}\in\Omega^3(M_\ell)$ on~$M_\ell$ by interpolating between
the torsion-free $G_2$-structure~$\bar\phy$ of Theorem~\ref{Thm2.1}
for~$w=0$ and the closed $G_2$-structure~$\phy_\ell$
for~$w\in\bigl[\frac12,1\bigr]$, just as we did for~$g_{\ell,w}$.
Because the metric defined by~$\phy_{\ell,w}$ depends in a nonlinear way
on~$\phy_{\ell,w}$, it will not agree with~$g_{\ell,w}$.
Instead, both metrics will differ by an error of the order~$O(e^{-c\ell})$
by Theorem~\ref{Thm2.1}.
We conclude that Clifford multiplication by~$\phy_{\ell,w}$ will have
an eigenspace of real dimension~$1$ for an eigenvalue close to~$\pm 7$.
We project~$s_\ell$ to it and normalise to define~$s_{\ell,w}$.
Then~$s_{\ell,w}$ will be the defining spinor for~$\phy_\ell$
on the subset~$M\setminus(X\times([-2,-1]\cup[1,2]))$ where~$\phy_{\ell,w}$
is torsion-free.

We let~$\nabla^{SM,\ell,w}$ denote the connection on~$SM$ induced from
the Levi-Civita connection on~$(M,g_{\ell,w})$ for each~$w$.
From the construction of~$g_{\ell,w}$ and Theorem~\ref{Thm2.1},
we conclude that
\begin{equation*}
  \bigl\|\nabla^{SM,\ell,w}s_{\ell,w}\bigr\|_{C^k}=O\bigl(e^{-c\ell}\bigr)
\end{equation*}
for some constant~$c>0$.

Recall that by Remark~\ref{Rem2.2}~\ref{2.g.1},
the metric~$g_\ell$ restricted to~$M_{\pm,\ell+1}$ descends from a product metric
on~$\widetilde M_{\pm,\ell+1}=V_{\pm,\ell+1}\times S^1_{\xi_\pm}$.
In particular, the unit tangent field~$\del_{\anglex_\pm}$ to the exterior
circle~$S^1_{\xi_\pm}$ induces a parallel vector field on~$M_{\pm,\ell+1}$
for~$w\ge\frac12$.
We still denote it by~$\del_{\anglex_\pm}$.
The $3$-form~$\phy_\ell$ of~\eqref{2.6a}, \eqref{2.7}
is of product type and hence parallel in the direction of~$\del_{\anglex_\pm}$.
Because~$s_{\ell,w}$ is a unit spinor that depends only on~$\phy_\ell$
for~$w\ge\frac12$,
we conclude that
\begin{equation*}
  \nabla^{SM,\ell,w}_{\del_{\anglex_\pm}}s_{\ell,w}=0
  \qquad\text{on }M_{\pm,\ell+1}\times\bigl[\tfrac12,1\bigr]\;.
\end{equation*}

Let~$D'_{\ell,w}$ denote the geometric spin Dirac operator of~$M_{\ell,w}$.
By Remark~\ref{Rem2.2}, the metric~$g_\ell$ has local holonomy in~$G_2$
except over~$X\times([-2,-1]\cup[1,2])$.
We have assumed that~$s_{\ell,w}$ defines the $G_2$-structure~$\phy_\ell$
outside~$X\times([-2,-1]\cup[1,2])$, hence
\begin{equation*}
  D'_{\ell,w}s_{\ell,w}|_{M_\ell\setminus X\times([-2,-1]\cup[1,2])}=0
\end{equation*}
for all~$w\ge\frac12$.
Because~$\bar g_\ell=g_{\ell,0}$ is a $G_2$-metric,
we also have~$D'_{\ell,0}s_{\ell,0}=0$.
Let us write
	$$D'_{\ell,w}s_{\ell,w}=f_{\ell,w}\cdot s_{\ell,w}+r_{\ell,w}\;,$$
where~$f_{\ell,w}\in C^\infty(M)$ and~$r_{\ell,w}\in\Gamma(\Spinor M)$
with~$r_{\ell,w}\perp s_{\ell,w}$ pointwise for all~$w$.
Then~$f_{\ell,w}$ and~$r_{\ell,w}$ vanish for~$w=0$,
and outside~$X\times([-2,-1]\cup[1,2])$ if~$w\in\bigl[\frac12,1\bigr]$.
Now define
\begin{equation}\label{ModifiedDirac1}
  D''_{\ell,w}=D'_{\ell,w}-\<\punkt,s_{\ell,w}\>\,(f_{\ell,w}\cdot s_{\ell,w}+r_{\ell,w})-\<\punkt,r_{\ell,w}\>\,s_{\ell,w}\;.
\end{equation}
Then~$D''_{\ell,w}$ is a self-adjoint differential operator with the symbol
of a Dirac operator, and~$s_{\ell,w}\in\ker(D''_{\ell,w})$ for all~$w$,
and~$D''_{\ell,0}=D'_{\ell,0}$ because~$s_{\ell,0}\in\ker(D'_{\ell,0})$.

We need another deformation for~$w\in\bigl[\frac 12,1\bigr]$.
Note that~$\del_{\anglex_\pm}$ is parallel with respect to~$\nabla^{TM,\ell,w}$
for~$w\ge\frac12$ and~$\pm t\ge -1$ by Remark~\ref{Rem2.2}~\ref{2.g.1}.
For~$w\ge\frac12$, let us decompose
\begin{equation*}
  r_{\ell,w}|_{M_{\pm,\ell}}=f^\pm_{\ell,w}\,c_{\anglex_\pm}s_{\ell,w}+r^\pm_{\ell,w}
\end{equation*}
with~$f^\pm_{\ell,w}=\<r_{\ell,w},c_{\anglex_\pm}s_{\ell,w}\>|_{M_{\pm,\ell}}$,
such that~$r^\pm_{\ell,w}$ is pointwise perpendicular to~$s_{\ell,w}$
and~$c_{\anglex_\pm}s_{\ell,w}$.
Again, both~$f^\pm_{\ell,w}$ and~$r^\pm_{\ell,w}$ are supported
on~$X\times(\pm[1,2])$ and of order~$O(e^{-c\ell})$.
Now, choose a cutoff function~$\rho$ supported in~$\bigl(\frac12,\infty\bigr)$
with~$1-\rho$ supported in~$(-\infty,1)$ and put
\begin{equation}\label{ModifiedDirac2}
  D_{\ell,w}|_{M_\pm}=D''_{\ell,w}|_{M_\pm}
  +\rho(w)\,\Bigl(\<\punkt,c_{\anglex_\pm}s_{\ell,w}\>\,c_{\anglex_\pm}(f_{\ell,w}\,s_{\ell,w}+r^\pm_{\ell,w})
	+\<\punkt,c_{\anglex_\pm}r^\pm_{\ell,w}\>\,c_{\anglex_\pm}s_{\ell,w}\Bigr)\;.
\end{equation}
Again, we obtain a smooth family of selfadjoint differential operators
with the principal symbol of a Dirac operator.
Note that we still have
\begin{equation}\label{3.21}
  D_{\ell,w}s_{\ell,w}=0
\end{equation}
for all~$w$.
On the strictly cylindrical piece~$Z$,
we still have~$D_{\ell,w}=D''_{\ell,w}$ for all~$w\ge\frac12$.

We also have
\begin{equation}\label{3.22}
  D_{\ell,1}|_{M_{\pm,\ell+1}}(c_{\anglex_\pm}s_{\ell,w})=0\;.
\end{equation}
To see this,
note that
\begin{equation}\label{3.22a}
  D'_{\ell,1}|_{M_{\pm,\ell+1}}
  =c_{\anglex_\pm}\,\bigl(\nabla^{SM,w}_{\del_{\anglex_\pm}}+A_{V_\pm,\ell}\bigr)\;,
\end{equation}
where~$A_{V_\pm,\ell}$ can be identified with the geometric Dirac operator
on~$(V_{\pm,\ell+1},g_{V_\pm,\ell})$.
Because~$\nabla^{SM,w}_{\del_{\anglex_\pm}}s_{\ell,w}=0$ and~$A_{V_\pm,\ell}$
anticommutes with~$c_{\anglex_\pm}$, we get
\begin{align*}
  D'_{\ell,1}(c_{\anglex_\pm}s_{\ell,w}|_{M_{\pm,\ell+1}})
  &=-c_{\anglex_\pm}D'_{\ell,1}s_{\ell,w}|_{M_{\pm,\ell+1}}
  =-f_{\ell,w}c_{\anglex_\pm}s_{\ell,w}-c_{\anglex_\pm}r_{\ell,w}\\
  &=-f_{\ell,w}c_{\anglex_\pm}s_{\ell,w}+f^\pm_{\ell,w}s_{\ell,w}-c_{\anglex_\pm}r^\pm_{\ell,w}\;.
\end{align*}
Because~$f^\pm_{\ell,w}=\<c_{\anglex_\pm}s_{\ell,w},r_{\ell,w}\>$,
inserting the above and~\eqref{ModifiedDirac1} into~\eqref{ModifiedDirac2}
gives
\begin{equation*}
  D_{\ell,1}|_{M_{\pm,\ell+1}}(c_{\anglex_\pm}s_{\ell,w})
  =D'_{\ell,1}(c_{\anglex_\pm}s_{\ell,w})
  -\bigl\<c_{\anglex_\pm}s_{\ell,w},r_{\ell,w}\bigr\>\,s_{\ell,w}
  +c_{\anglex_\pm}(f_{\ell,w}\,s_{\ell,w}+r^\pm_{\ell,w})\\
  =0\;.
\end{equation*}
Finally, for each~$k\ge 0$, we have
\begin{equation*}
  \norm{D_{\ell,w}-D_{\ell,0}}_{C^k}=O\bigl(e^{-c\ell}\bigr)\;.
\end{equation*}

\subsection{Spectral Symmetry on \texorpdfstring{$M_\pm$}{M+-}}\label{SpecSymm}
On the manifolds~$(M_{\pm,\ell},g_\ell)$,
we consider the restriction $\dirac_{M_\pm,\ell}=\dirac_{\ell,1}|_{M_{\pm,\ell}}$
with modified APS-boundary conditions.
As in~\cite{KiLe}, we let~$\gamma=c_t$ denote Clifford multiplication
by the vector~$\del_t$ and write
\begin{equation}\label{3.22b}
  \dirac_{\ell,1}|_{t\in[-1,1]}
  =\gamma\,\biggl(\frac\del{\del t}+A\biggr)\;.
\end{equation}
We therefore have to choose Lagrangians~$\lagr_{D_\pm}\subset\ker A$.

\begin{rmk}\label{rem:par}
  We can identify~$A$ with the geometric Dirac operator~$D_X$
  on~$X=\K\times T^2$. In particular, $\dim(\ker A)=4$,
  and~$\ker A$ is the space of parallel spinors on~$X$
  because~$X$ is scalar flat.
  
  Here, we identify~$\ker A$ with the space of parallel spinors
  over~$X\times(-1,1)$.
  Starting from the parallel spinor~$s=s_{\ell,1}|_{X\times(-1,1)}$,
  we claim that~$\ker A$ is spanned~$s$,
  $c_{\anglen_-}s$, $c_{\anglex_-}s$ and~$\gamma s$.
  This follows because~$\del_{\anglen_-}$, $\del_{\anglex_-}$ and~$\del_t$
  span the space of parallel vector fields on~$X\times(-1,1)$,
  and Clifford multiplication gives a parallel isomorphism from
  the tangent bundle to the subbundle of the spinor bundle
  that is perpendicular to~$s$.
\end{rmk}

Let~$\lagr_{D_\pm}\subset\ker A$ be the Lagrangian subspace
spanned by those $A$-harmonic spinors on~$\X$ that extend
to $\dirac_{\ell,1}$-harmonic spinors on~$M_\pm$,
as proposed in~\cite{KiLe}.
By~\eqref{3.21}, \eqref{3.22},  we know that~$s$, $c_{\anglex_\pm}s\in \lagr_{D_\pm}$.
Because~$\dim \lagr_{D_\pm}=2$, we conclude that
\begin{align}
  \begin{split}\label{3.23}
    \lagr_{D_-}&=\Span\{s,c_{\anglex_-}s\}\\ %
    \text{and}\qquad
    \lagr_{D_+}&=\Span\{s,c_{\anglex_+}s\}\;. %
  \end{split}
\end{align}

\begin{prop}\label{prop3.1}
  Assume that~$\Gamma_\pm\cong\Z/k_\pm$ with~$k_\pm\in\{1,2\}$.
  Then
  \begin{equation*}
    \eta_{\APS}\bigl(D_{M_\pm,\ell};\lagr_{D_\pm}\bigr)=0\;.
  \end{equation*}
\end{prop}

We proceed as in the proof of Proposition~\ref{prop2.3}.
Again, in~\cite{GN} we also allow~$k_\pm>2$.
In that case, it is possible
that~$\eta_{\APS}\bigl(D_{M_\pm,\ell};\lagr_{D_\pm}\bigr)\ne0$.

\begin{proof}
Let~$\kappa_-$ be the involution of~$M_-$
introduced in Section~\ref{Abs2.3}.
Because~$\Spinor M_-$ is the pullback of the
spinor bundle~$\Spinor V_-$ on~$V_-$,
we can lift~$\kappa_-$ to the spinor bundle~$\Spinor M_-$ by
	$$\Spinor M_-|_{(v_-,x)}=\Spinor V_-|_x\;\owns\;\sigma
	\longmapsto \bar \kappa_-(\sigma)=c_{\anglex_-}\sigma
	\;\in\;\Spinor V_-|_x=\Spinor M_-|_{(-v_-,x)}\;.$$
Because~$c_{\anglex_-}$ anticommutes with~$A_{V_\pm,1}$ and
the reflection of~$S^1$ anticommutes with~$\frac\del{\del {\anglex_-}}$,
we see from~\eqref{3.22a} that the geometric Dirac operator~$\dirac'_{\ell,1}$
anticommutes with~$\bar \kappa_-$.

To show that~$\bar\kappa_-$ anticommutes with the additional
terms in~\eqref{ModifiedDirac1}, \eqref{ModifiedDirac2},
we note that~$s_{\ell,1}|_{M_-}$ and~$D'_{\ell,1}s_{\ell,1}|_{M_-}$
do not depend on the variable~$\anglex_-$.
Then~$f_{\ell,1}$, $f^-_{\ell,1}$, $r_{\ell,1}$ and~$r^-_{\ell,1}$ are independent of~$\anglex_-$
as well,
so the additional terms in~\eqref{ModifiedDirac1}, \eqref{ModifiedDirac2}
commute with a reflection of~$S^1$.
Because~$c_{\anglex_-}$ is skew-adjoint with~$c_{\anglex_-}^2=-1$,
one checks that~$D_{\ell,1}$ anticommutes with~$\bar\kappa_-$.

On the other hand, $\bar \kappa_-$ also anticommutes with~$\gamma=c_t$,
and hence, $\bar \kappa_-$ commutes with~$A$ by~\eqref{3.22b}.
From~\eqref{3.23} it is clear that~$\bar\kappa_-$ preserves~$\lagr_{D_-}$.
Then~$\bar \kappa_-$ preserves the modified APS-boundary conditions on~$M_-$,
that is, the space of sections of~$\Gamma(SM_-)$ whose restriction to~$\X$
is perpendicular to the space
spanned by~$\lagr_{D_-}$ and the eigenspinors of~$A$
with positive eigenvalues.
Because~$\bar\kappa_-$ anticommutes with~$\dirac_{M_-,\ell}$,
we have established the Proposition.
\end{proof}

\subsection{Small eigenvalues}\label{Sec:SmallEv}
Before we prove a gluing formula for the spinorial Dirac operator,
we show that there are no small eigenvalues except~$0$.
This will be needed both for the gluing formula and for the comparison
of~$\eta(D_{M,\ell,1})$ and~$\eta(D_M)$.
We will first analyse the spectrum of the Dirac operators on~$M_\pm$
under strict APS boundary conditions,
using Cheeger's estimate for the scalar Laplacian.
In a second step,
we will %
prove a lower bound on the first positive Dirac eigenvalue on~$M$.
The estimates in this section still depend on the geometry of an extra
twisted connected sum.

To state these estimates,
we fix~$\ell$, and we replace the cylindrical part~$X\times[-1,1]$ of~$M_\ell$
from Remark~\ref{Rem2.2}~\ref{2.g.3}
by a cylinder of length~$2r$, obtaining~$M_{\ell,r}$.
Write~$Z_r=X\times[-r,r]\subset M_{\ell,r}$ for the new cylindrical part.
Let~$M_{\pm,\ell,r}\subset M_{\ell,r}$ be extensions of~$M_{\pm,\ell}$
such that~$M_{+,\ell,r}\cap M_{-,\ell,r}=Z_r$,
and let~$D_{M_\pm,\ell,r}$ denote the operators induced by~$D_{\ell,1}$
on~$M_{\pm,\ell,r}$.
In particular, the restriction~$D_{Z_r}=D_{M,\ell,r}|_{Z_r}=D_{M_\pm,\ell,r}|_{Z_r}$
is still described by~\eqref{3.22b}.

To treat eigenspinors of~$D_{\ell,1}$ on~$Z_r$, we use separation of variables,
as follows.
In~\eqref{3.22b}, the operator~$A$ anticommutes with~$\gamma$.
This implies that~$\gamma\frac\del{\del t}$ and~$D_{Z_r}$ preserve
each eigenspace of~$A^2$.
We start with nonzero eigenvalues of~$A$, leaving~$\ker A$ for later.
Let~$\mu>0$ be an eigenvalue of~$A$ and let~$h_\mu$
be a normalised $\mu$-eigenspinor of~$A$ on~$X$.
Then~$\gamma h_\mu$ is a normalised $(-\mu)$-eigenspinor.
Assume that~$2\abs\lambda$ is smaller than the smallest
positive eigenvalue of~$A$, in particular, $2\abs\lambda<\mu$.
Let~$t\in[-r,r]$ denote the cylinder coordinate.
Then the $\lambda$-eigenspace of the operator~$D_{Z_r}$
acting on spinors that are linear combinations of~$h_\mu$ and~$\gamma h_\mu$
in each fibre of~$Z_r\to[-r,r]$ is spanned by
\begin{align}\label{eq:HpmrLambdaMu}
  \begin{split}
    H^-_{r,\lambda,\mu}
    &=\Bigl(\bigl(\mu+\sqrt{\mu^2-\lambda^2}\bigr)\,h_\mu
    +\lambda\,\gamma h_\mu\Bigr)\,e^{-(t+r)\sqrt{\mu^2-\lambda^2}}\;,\\
    H^+_{r,\lambda,\mu}
    &=\Bigl(\lambda\,h_\mu
    +\bigl(\mu+\sqrt{\mu^2-\lambda^2}\bigr)\,\gamma h_\mu\Bigr)
    \,e^{(t-r)\sqrt{\mu^2-\lambda^2}}
  \end{split}
\end{align}
because of~\eqref{3.22b},
see~\cite[(3.6), (3.7)]{BuGlu}.

Let~$H=L^2(X;SX)$ denote the space of $L^2$-spinors on~$X$.
We equip it with the symplectic structure induced by~$\gamma=c_t$.
Let~$F^+(A)$ denote the closure of the sum of all positive eigenspaces of~$A$,
then~$F^+(A)$ is an isotropic subspace of~$H$
with annihilator~$F^+(A)\oplus\ker A$.
We also put $F^-(A)=\gamma F^+(A)=F^+(-A)$.

\begin{prop}\label{prop:nonres}
  There exist constants~$c>0$ and~$r_0>0$ such that
  for all~$\ell\gg 1$ sufficiently large, all~$r\ge r_0$
  and~$\abs\lambda<\frac c{\ell+r}$,
  there exists no $\lambda$-eigenspinor~$\psi_\pm$ of~$D_{M_\pm,\ell,r}$
  with~$\psi_\pm|_{\del M_{\pm,\ell,r}}\in F^\mp(A)$.
\end{prop}

In other words, the operator~$D_{M_\pm,\ell,r}$ has no small eigenvalues
under strict APS boundary conditions.
In particular, it has trivial kernel.
Such operators are called nonresonant, see~\mbox{\cite[Definition~4.8]{Nico}.}
Note that~$c$ and~$r_0$ do not depend on~$\ell$.

\begin{proof}
  By symmetry, it suffices to consider~$D_{M_-,\ell,r}$.
  Let~$\psi_-$ be a normalised $\lambda$-eigenspinor of~$D_{M_-,\ell,r}$
  with~$\psi_-|_{\del M_{-,\ell,r}}\in F^+(A)$.
  Let~$t\in[-r,r]$ denote the cylinder coordinate.
  We assume that~$2\abs\lambda$ is smaller than the smallest positive
  eigenvalue~$\mu_0$ of~$A$.
  By~\eqref{3.22b} and~\eqref{eq:HpmrLambdaMu}, we may write
  \begin{equation*}
    \psi_-|_{Z_r}
    =\sum_{\Spec A\owns\mu>0}a_\mu\Biggl(H^-_{r,\lambda,\mu}
    -\frac{\lambda\,e^{-2r\sqrt{\mu^2-\lambda^2}}}{\mu+\sqrt{\mu^2-\lambda^2}}
    \,H^+_{r,\lambda,\mu}\Biggr)\;,
  \end{equation*}
  where~$a_\mu\in\R$.
  The coefficient of~$H^+_{r,\lambda,\mu}$ has been chosen such that
  along~$\del M_{-,\ell,r}=X\times\{r\}$,
  the component in~$F^-(A)$ vanishes.
  If~$r\gg 1$ is sufficiently large, it is not hard to check that
  \begin{equation}\label{eq:psimEst1}
    \norm{\psi_-}_{L^2(X\times\{t\})}
    =O\Bigl(e^{\sqrt{\mu_0^2-\lambda^2}\,(1-r-t)}\Bigr)\;.
  \end{equation}

  By equation~\eqref{3.22b}
  and partial integration,
  \begin{align}
    \begin{split}\label{eq:PartIntSL}
      \lambda^2\norm{\psi_-}_{L^2}^2
      &=\norm{D_{M_-,\ell,r}\psi_-}_{L^2}^2\\
      &=\norm{\nabla\psi_-}_{L^2}^2
      +\<D_{M_-,\ell,r}^2-\nabla^*\nabla\psi_-,\psi_-\>_{L^2}
      -\<\gamma D_{M_-,\ell,r}\psi_-+\del_t\psi_-,\psi_-\>_{L^2(\del Z_r)}\\
      &=\norm{\nabla\psi_-}_{L^2}^2+\<\mathcal R_\ell\psi_-,\psi_-\>_{L^2}
      +\<A\psi_-,\psi_-\>_{L^2(X\times\{r\})}\;.
    \end{split}
  \end{align}
  Here, the term~$\mathcal R_\ell=D_{\ell,1}^2-\nabla^*\nabla$
  is of order~$O(e^{-\ell/C})$
  for some~$C>0$ by our construction~\eqref{ModifiedDirac1},
  \eqref{ModifiedDirac2} of~$D_{M_\pm,\ell,r}$
  and the Schr\"odinger-Lichnerowicz formula for the unperturbed Dirac operator,
  using that~$g_\ell$ has small scalar curvature.
  Because we assumed that~$\psi_-|_{X\times\{r\}}\in F^+$,
  we have~$\<A\psi_-,\psi_-\>_{L^2(X\times\{r\})}\ge 0$, and hence
  \begin{equation}\label{eq:psimEst2}
    \norm{\nabla\psi_-}_{L^2}^2
    \le\lambda^2\norm{\psi_-}_{L^2}^2+O\bigl(e^{-\ell/C}\bigr)\;.
  \end{equation}
  
  Choose a smooth cutoff function~$\rho\colon\R\to[0,1]$
  with~$\rho|_{(-\infty,0)}\equiv 0$ and~$\rho|_{(1,\infty)}\equiv 1$.
  We consider the function~$f=\rho(r-t)\norm{\psi_-}$
  with~$f|_{\del M_{-,\ell,r}}=0$.
  Using the classical pointwise
  estimate~$\bigl\|d\norm\psi\bigr\|\le\norm{\nabla\psi}$
  from~\cite[(2.2)]{HSU} together with~\eqref{eq:psimEst1}, \eqref{eq:psimEst2},
  we see that
  \begin{equation*}
    \norm{df}_{L^2}\le\norm f_{L^2}
    \,\Bigl(\lambda^2+O\bigl(e^{-\ell/C}\bigr)+O\bigl(e^{-r/C}\bigr)\Bigr)
  \end{equation*}
  for some constant~$C>0$.
  This gives an upper bound for the first Dirichlet eigenvalue~$\lambda_1$
  of the scalar Laplacian on~$M_{-,\ell,r}$.
  On the other hand,
  Cheeger's estimate~\cite{Cheeger} says that
  \begin{equation*}
    \frac{\vol(X)^2}{4\vol(M_{-,\ell,r})^2}\le\lambda_1\;.
  \end{equation*}
  Because the volume of~$M_{-,\ell,r}$ is approximately linear in~$\ell+r$,
  whereas the boundary~$X$ has constant volume,
  there exists a constant~$C'>0$ such that
  \begin{equation*}
    \frac{C'}{(\ell+r)^2}\le\frac{\vol(X)^2}{4\vol(M_{-,\ell,r})^2}\le\lambda_1
    \le C\Bigl(\lambda^2+O\bigl(e^{-\ell/C}\bigr)+O\bigl(e^{-r/C}\bigr)\Bigr)\;.
    \qedhere
  \end{equation*}
\end{proof}

For each~$h_0\in\ker A$, we have a $\lambda$-eigenspinor
of~$D_{Z_r}$ of the form~$e^{-t\lambda\gamma}\,h_0$.
Recall that~$\ker A$ is
spanned by~$s$, $c_{u_-}s$, $c_{u_+}s$ and~$\gamma s$,
where~$\gamma=c_t$, see Remark~\ref{rem:par}.
Because~$\gamma c_{v_-}s$ is perpendicular to~$s$, $c_{v_-}s$ and~$c_ts$
by the Clifford relation, we conclude that~$\gamma c_{v_-}s=\pm c_{u_-}s$.
The actual sign is a matter of convention and plays no major role
in the following.
Because the coordinate transformation in~\eqref{2.5}
is orientation reversing, we can deduce that~$\gamma c_{v_+}s=\mp c_{u_+}s$,
with the opposite sign as above.
In particular, we have
\begin{equation}\label{eq:cvpms}
  c_{v_-}s=e^{\mp\thet\gamma}\,c_{v_+}s
  \;.
\end{equation}

\begin{prop}\label{prop:smallev}
  There exists a constant~$c>0$ such that for all~$\ell\gg 1$
  sufficiently large and all~$r\ge\ell$,
  the operator~$D_{M,\ell,r}$ has no $\lambda$-eigenspace
  with~$\abs\lambda<\frac c{\ell+r}$ apart from its kernel,
  which is spanned by~$s$.
\end{prop}

Clearly, this statement fails %
if we allow gluing angles~$\thet\in\pi\Z$,
because then~$c_{v_-}s=c_{v_+}s$ on~$Z_r$,
giving rise to an additional very small eigenvalue.
Recall that we have excluded~$\thet\in\pi\Z$ because that would result in
``untwisted connected sums'' with infinite fundamental group,
which is impossible if the holonomy group is isomorphic to~$G_2$.

\begin{proof}
  It is clear by~\eqref{3.21} that~$s\in\ker D_{M,\ell,r}$.
  Assume that~$\psi$ is a $\lambda$-eigenspinor of~$D_{M,\ell,r}$
  with~$\<\psi,s\>_{L^2}=0$, $\norm\psi_{L^2}=1$,
  and~$\abs\lambda<\frac c{\ell+r}$ for some constant~$c$ that
  we will determine later.
  We will construct spinors~$\psi_\pm$ on~$M_{\pm,\ell,r}$
  that vanish at the boundary
  and such that~$\norm{D_{M_\pm,\ell,r}\psi_\pm}$ is small,
  contradicting Proposition~\ref{prop:nonres}.
  Throughout the proof, $c>0$ will denote the constant
  in the Proposition,
  which may differ from the one in Proposition~\ref{prop:nonres}.
  On the other hand, $C>0$ may change from line to line.

  Using the eigenspinors of~\eqref{eq:HpmrLambdaMu},
  we may decompose
  \begin{equation}\label{eqn:PsiDecomp}
    \psi|_{Z_r}=\psi_0
    +\sum_{\Spec A\owns\mu> 0}\bigl(a_\mu\,H^-_{r,\lambda,\mu}
    +b_\mu\,H^+_{r,\lambda,\mu}\bigr)\;,
  \end{equation}
  where~$\psi_0|_{X\times\{t\}}\in\ker A$ for each~$t\in[-r,r]$,
  and~$a_\mu$, $b_\mu\in\R$.
  Then there exist real numbers~$a$, $b$, $e$, $f$ such that
  \begin{equation*}
    h_0=\psi_0|_{X\times\{r\}}
    =(a+b\gamma)s|_{X\times\{r\}}
    +(e+f\gamma)c_{v_-}s|_{X\times\{r\}}\;.
  \end{equation*}
  Because~$D_{Z_r}\psi_0=\lambda\psi_0$, we have
  \begin{equation*}
    \psi_0=e^{(r-t)\lambda\gamma}\,h_0\;.
  \end{equation*}
  Using~\eqref{eq:cvpms}, we conclude that
  \begin{equation}\label{eq:psipmr}
    \psi_0|_{X\times\{-r\}}
    =e^{2r\lambda\gamma}(a+b\gamma)s|_{X\times\{-r\}}
    +e^{(2r\lambda\mp\thet)\gamma}(e+f\gamma)c_{v_+}s|_{X\times\{-r\}}\;.
  \end{equation}

  Because the parallel spinor~$s$ has pointwise length~$1$,
  we have~$\norm s_{L^2(M_{\pm,\ell,r})}=\sqrt{\vol(M_{\pm,\ell,r})}$,
  which is of the order of~$\sqrt{\ell+r}$.
  On the other hand, $\norm s_{L^2(X)}=\sqrt{\vol X}$
  is independent of~$\ell$ and~$r$.
  Recall that for any Dirac operator~$D$ and two spinors~$v$, $w$,
  the function~$\<Dv,w\>-\<v,Dw\>$ is the divergence of
  the vector field dual to the 1-form~$\<c_{\punkt}v,w\>$.
  By our construction of~$D_{\ell,1}$ in~\eqref{ModifiedDirac1},
  \eqref{ModifiedDirac2}, this still holds for our modified Dirac operator.
  By partial integration, we estimate
  \begin{equation}\label{eq:kerAest1}
    \abs b=\frac{\bigl|\<\psi,\gamma s\>_{L^2(X\times\{r\})}\bigr|}{\norm s^2_{L^2(X)}}
     =\frac{\bigl|\<D_{M_-,\ell,r}\psi,s\>_{L^2(M_{-,\ell,r})}\bigr|}
    {\norm s^2_{L^2(X)}}
    \le\abs\lambda\,\frac{\sqrt{\vol(M_{-,\ell,r})}}{\vol(X)}
    <\frac{Cc}{\sqrt{\ell+r}}\;,
  \end{equation}
  where~$c$ is the constant in the Proposition and~$C>0$
  depends on the volumes of~$M_{\pm,\ell,r}$ and~$X$.
  Using~$c_{v_-}s\in\ker D_{M_-,\ell,r}$ instead of~$s$,
  we get a similar estimate for~$\abs f$.
  
  Using~\eqref{eq:psipmr} and~$c_{v_+}s\in\ker D_{M_+,\ell,r}$,
  we also find that
  \begin{equation*}
    \abs{\sin(2r\lambda\mp\thet)\,e+\cos(2r\lambda\mp\thet)\,f}
    =\frac{\bigl|\<D_{M_+,\ell,r}\psi,c_{v_+}s\>_{L^2(M_{+,\ell,r})}\bigr|}
    {\norm{c_{v_+}s}^2_{L^2(X)}}
    <\frac{Cc}{\sqrt{\ell+r}}\;.
  \end{equation*}
  Note that~$\thet\notin\pi\Z$ by assumption.
  If we choose~$c>0$ small enough, we can make sure
  that~$\abs{\sin(2r\lambda\mp\thet)}$ is bounded away from~$0$
  for all~$\lambda$ with~$\abs\lambda<\frac c{\ell+r}$.
  This way, we also get~$\abs e<Cc/\sqrt{\ell+r}$.
  However, we cannot expect a similar estimate for~$\abs a$
  simply because~$s\in\ker D_{M,\ell,r}$ and we have not used
  the condition~$\<\psi,s\>_{L^2(M)}=0$ yet.
  To summarize,
  if we choose~$c>0$ small enough, then there exists~$C>0$ such that
  \begin{equation}\label{eq:abefEst}
    \begin{aligned}
      \max\{\abs b,\abs e,\abs f\}&<\frac{Cc}{\sqrt{\ell+r}}\\
      \text{and}\qquad
      \abs a=\frac{\bigl|\<\psi,e^{(r-t)\lambda\gamma}s\>_{L^2(Z_r)}\bigr|}
           {\norm s^2_{L^2(Z_r)}}
      &\le\frac{\norm\psi_{L^2}}{\norm s_{L^2}}=\frac1{\sqrt{\vol(Z_r)}}
      <\frac C{\sqrt{r}}\;.
    \end{aligned}
  \end{equation}
  
  Choose a smooth cutoff function~$\rho\colon\R\to[0,1]$
  with~$\rho|_{(-\infty,0)}\equiv 0$ and~$\rho|_{(1,\infty)}\equiv 1$.
  Similar as in~\cite[Section~4.6]{BuGlu},
  starting from the decomposition~\eqref{eqn:PsiDecomp},
  we define
  \begin{align*}
    \psi_\pm|_{M_{\pm,\ell,r}\setminus Z_r}
    &=\psi|_{M_{\pm,\ell,r}\setminus Z_r}-as\;,\\
    \psi_-|_{Z_r}
    &=\rho\Bigl(\frac{r-t}{2r}\Bigr)\,(\psi_0-as)\\
    &\qquad
    +\sum_{\Spec A\owns\mu>0}\Bigl(a_\mu\rho(r-t)H^-_{r,\lambda,\mu}
    +b_\mu\bigl(1-\rho(r+t)\bigr)H^+_{r,\lambda,\mu}\Bigr)\;,\\
    \text{and}\qquad
    \psi_+|_{Z_r}
    &=\biggl(1-\rho\Bigl(\frac{r-t}{2r}\Bigr)\biggr)\,(\psi_0-as)\\
    &\qquad
    +\sum_{\Spec A\owns\mu>0}\Bigl(a_\mu\bigl(1-\rho(r-t)\bigr)H^-_{r,\lambda,\mu}
    +b_\mu\rho(r+t)H^+_{r,\lambda,\mu}\Bigr)\;.
  \end{align*}
  In particular,
  the spinors~$\psi_\pm$ are smooth and vanish at the boundary
  of~$M_{\pm,\ell,r}$.
  If we extend~$\psi_\pm$ by~$0$ to all of~$M$,
  then~$\psi-as=\psi_++\psi_-$.
  Because~$\psi$ is normalised and~$\<\psi,s\>_{L^2(M)}=0$ by assumption,
  we may conclude that
  \begin{equation}\label{eqn:wlog}
    \max\bigl\{\norm{\psi_-}_{L^2},\norm{\psi_+}_{L^2}\bigr\}
    \ge\frac12\;.
  \end{equation}

  Because~$\psi$ is a $\lambda$-eigenspinor and~$s$ is parallel,
  we get
  \begin{multline}\label{eq:DpsiEst}
    \norm{D_{M_\pm,\ell,r}\psi_\pm}_{L^2(M_{\pm,\ell,r})}
    \le\abs\lambda\norm{\psi}_{L^2(M_{\pm,\ell,r})}\\
    +\frac 1{2r}\,\norm{\rho'\Bigl(\frac{r-t}{2r}\Bigr)(\psi_0-as)}_{L^2(Z_r)}
    +C\sum_{\Spec A\owns\mu>0}e^{(1-2r)\sqrt{\mu^2-\lambda^2}}(\abs{a_\mu}+\abs{b_\mu})\;.
  \end{multline}
  The contribution of the nonzero eigenvalues of~$A$ on the cylinder
  is of order~$O(e^{-r/C})$.
  For the kernel, we have the pointwise estimate
  \begin{align*}
    \norm{\psi_0-as}^2
    &=\bigl\|\bigl(e^{(r-t)\lambda\gamma}(a+b\gamma)-a\bigr)s
      +e^{(r-t)\lambda\gamma}(e+f\gamma)c_{v_-}s\bigr\|^2\\
    &=\vol(X)\,\Bigl(2a^2\,\bigl(1-\cos((r-t)\lambda)\bigr)
      -2ab\,\sin((r-t)\lambda)+b^2+e^2+f^2\Bigr)\\
    &<\frac {Cc^2}{r}
  \end{align*}
  for some constant~$C>0$
  by~\eqref{eq:abefEst} and our assumption on~$\lambda$.
  We can now control the second term on the right hand side
  of~\eqref{eq:DpsiEst} and obtain
  \begin{equation*}
    \norm{D_{M_\pm,\ell,r}\psi_\pm}_{L^2(M_{\pm,\ell,r})}
    \le\abs\lambda\norm{\psi}_{L^2(M_{\pm,\ell,r})}
    +\frac{Cc}r+O\bigl(e^{-r/C}\bigr)\;.
  \end{equation*}

  By~\eqref{eqn:wlog},
  we may assume that~$\norm{\psi_-}\ge\frac12$.
  We estimate the Rayleigh quotient of~$\psi_-$ by
  \begin{equation*}
    \frac{\norm{D_{M_-,\ell,r}\psi_-}_{L^2(M_{-,\ell,r})}}
         {\norm{\psi_-}_{L^2(M_{-,\ell,r})}}
    \le\frac{Cc}r+O\bigl(e^{-r/C}\bigr)\;.
  \end{equation*}
  But this implies that for~$c>0$ sufficiently small and~$r\ge\ell$,
  even under Dirichlet boundary conditions,
  the operator~$D_{M_-,\ell,r}$ has an eigenvalue
  that is smaller than the smallest eigenvalue
  allowed by Proposition~\ref{prop:nonres}.
  Hence, there cannot be a $\lambda$-eigenspinor of~$D_{M,\ell,r}$
  with~$\abs\lambda<\frac c{\ell+r}$ and~$\<\psi,s\>_{L^2}=0$.
\end{proof}

\subsection{A Gluing Formula}
In this section,
we derive a formula similar to~\cite[(8.32)]{KiLe} for the modified spin Dirac
operator~$D_{\ell,1}$ of an extra-twisted connected sum.
With the construction of~$D_{M,\ell,r}$ and the estimates from the previous
section, we can follow either Bunke~\cite{BuGlu} or Kirk-Lesch~\cite{KiLe}.
For simplicity, we have chosen the former.

We still consider the manifolds~$M_{\ell,r}$ and~$M_{\pm,\ell,r}$ introduced
in Section~\ref{Sec:SmallEv}.
Let~$Z'_r\cong Z_r$ be an additional copy of the cylindrical part.
Then let~$\D_1=D_{M,\ell,r}\oplus D_{Z'_r}$ denote the operator induced
by~$D_{M,\ell,r}$ on the disjoint union~$M_{\ell,r}\sqcup Z'_r$.
We use APS boundary conditions modified by the subspaces~$\lagr_{D_\pm}$
of~\eqref{3.23} at~$\{\mp r\}\times X\subset\del Z'_r$.
Recall the Maslov angle~$m_H(\lagr_{D_+},\lagr_{D_-})$ from Definition~\ref{BunkeDef}.
By a theorem of Lesch-Wojciechowski~\cite[Thm~2.1]{LW},
\begin{equation*}
  \eta_{\APS}(\D_1;\lagr_{D_-}\oplus \lagr_{D_+})
  =\eta(D_{M,\ell,r})-m_{\ker A}(\lagr_{D_+},\lagr_{D_-})\;.
\end{equation*}

Let~$\chi\colon[-1,1]\to[0,1]$ be a smooth, nowhere increasing function
with~$\supp\chi\subset\bigl[-1,\frac 12\bigr]$,
and such that $\chi^2(-t)+\chi^2(t)=1$.
In particular, $\supp(1-\chi)\subset\bigl[-\frac12,1\bigr]$.
We now regard~$t$ as the cylinder coordinate
on~$Z_r\subset M_{\ell,r}$ and~$Z'_r$,
and we extend~$\chi\bigl(\frac tr\bigr)\colon Z_r\to[0,1]$
by~$1$ on~$M_{\ell,r}^-\setminus Z_r$,
and by~$0$ on~$M_{\ell,r}^+\setminus Z_r$,
and we similarly extend~$\chi\bigl(-\frac tr\bigr)$.
Abusing notation,
we define an $L^2$-unitary transformation
\begin{equation*}
  U=
  \begin{pmatrix}
    \chi\bigl(\tfrac tr\bigr)&-\chi\bigl(-\tfrac tr\bigr)\\
    \chi\bigl(-\tfrac tr\bigr)&\chi\bigl(\tfrac tr\bigr)
  \end{pmatrix}
  \colon\Gamma(SM_{\ell,r})\oplus\Gamma(SZ'_r)
  \longrightarrow\Gamma(SM_{-,\ell,r})\oplus\Gamma(SM_{+,\ell,r})
\end{equation*}
as in~\cite[Section~3.2]{BuGlu},
and consider the operator
\begin{equation*}
  \D_0=U^*\bigl(D_{M_-,\ell,r}\oplus D_{M_+\ell,r}\bigr)U\;.
\end{equation*}
Put~$\trang(t)=\chi'(t)\chi(-t)+\chi'(-t)\chi(t)$.
Then~$\trang(t)\le 0$ because~$\chi$ is nonincreasing.
Define
\begin{equation}\label{eq:Gdef}
  G=\D_1-\D_0=\frac 1r
  \begin{pmatrix}
    0&-\trang\bigl(\tfrac tr\bigr)\,\gamma\\
    \trang\bigl(\tfrac tr\bigr)\,\gamma&0
  \end{pmatrix}\;.
\end{equation}
Then~$G$ is a selfadjoint operator
on~$\Gamma(SM_{\ell,r})\oplus\Gamma(SZ'_r)$.
It is of order~$0$, but not a local operator because it acts between
the spinor bundles on the two copies of~$Z_r$.
For~$z\in[0,1]$, define
\begin{equation}\label{eq:Dzdef}
  \D_z=\D_0+zG=D_{M,\ell,r}\oplus D_{Z'_r}-(1-z)G\;.
\end{equation}

\begin{thm}\label{thm3.4}
  For~$\ell\gg 1$ sufficiently large and~$r\ge\ell$, 
  on~$M_{\ell,r}=M_{-,\ell,r}\cup_{Z_r} M_{+,\ell,r}$
  we have
  \begin{equation*}
    \eta(D_{M,\ell,r})=\eta_{\APS}(D_{M_+\ell,r};\lagr_{D_+})
	+\eta_{\APS}(D_{M_-,\ell,r};\lagr_{D_-})+m_{\ker A}(\lagr_{D_+},\lagr_{D_-})\;.
  \end{equation*}
\end{thm}
  
\begin{proof}[Proof of Theorem~\ref{thm3.4}]
  By~\cite[Theorem~1.9]{BuGlu} and the remark after~\cite[Theorem~1.17]{BuGlu},
  it suffices to show that
  \begin{equation}\label{3.8}
    2\SF\bigl((\D_z)_{z\in[0,1]}\bigr)=\dim\ker \D_0-\dim\ker \D_1\;.
  \end{equation}
  Here, the spectral flow~$\SF\bigl((\D_z)_{z\in[0,1]}\bigr)$
  is defined as the net number of eigenvalues
  crossing the line~$\eps>0$ from below for~$\eps$ very small.

  We start by proving an analogue of~\eqref{eq:PartIntSL}
  for the operator~$\D_z$.
  We introduce the operator
  \begin{align}
    \begin{split}\label{eq:nablaz}
      \nabla^z&\colon\Gamma(SM_{\ell,r})\oplus\Gamma(SZ'_r)
      \to\Omega^1(M_{\ell,r};SM_{\ell,r})\oplus\Omega^1(Z'_r;SZ'_r)\\
      \text{with}\qquad
      \nabla^z
      &=\nabla+(1-z)\,U^*\,dU
      =\nabla+\frac{1-z}r
      \begin{pmatrix}
        0&\trang\bigl(\frac tr\bigr)\\-\trang\bigl(\frac tr\bigr)&0
      \end{pmatrix}\,dt
    \end{split}
  \end{align}
  Then~$\nabla^1$ is the spinor connection
  on~$\Gamma(SM_{\ell,r})\oplus\Gamma(SZ'_r)$,
  and~$\nabla^0=U^*\circ\nabla\circ U$ is isomorphic to the
  connection on~$\Gamma(SM_{+,\ell,r})\oplus\Gamma(SM_{i,\ell,r})$.
  For~$z<1$, the operator~$\nabla^z$ is nonlocal,
  as it mixes spinors on the two copies of~$Z_r$.
  However, its restriction to~$Z_r\sqcup Z'_r$
  becomes a metric connection if we regard~$SZ_r\oplus SZ'_r$
  as a vector bundle over a single copy of~$Z_r$.
  
  By~\eqref{eq:Gdef} and~\eqref{eq:nablaz}, we have
  \begin{equation*}
    G^2=(\D_1-\D_0)^2=(\nabla^1-\nabla^0)^*(\nabla^1-\nabla^0)
  \end{equation*}
  Then
  \begin{align*}
    \D_z^2
    &=z\,\D_1^2+(1-z)\,\D_0^2-z(1-z)\,(\D_1-\D_0)^2\\
    &=z\,\nabla^{1*}\nabla^1+(1-z)\,\nabla^{0*}\nabla^0
    -z(1-z)\,(\nabla^1-\nabla^0)^*(\nabla^1-\nabla^0)+\mathcal R_\ell\\
    &=\nabla^{z*}\nabla^z+\mathcal R_\ell\;,
  \end{align*}
  where~$\mathcal R_\ell=D_{\ell,1}^2-\nabla^*\nabla$ as in~\eqref{eq:PartIntSL}.
  Here we have used that~$\mathcal R$ and~$\trang$ have disjoint supports.

  Let now~$\psi$ be a normalised $\lambda$-eigenspinor of the operator~$\D_z$.
  We argue as in~\eqref{eq:PartIntSL} to obtain
  \begin{equation*}
    \lambda^2\norm{\psi}^2_{L^2}
    =\norm{\nabla^z\psi}^2_{L^2}+\<\mathcal R\psi,\psi\>_{L^2}
    +\<A\psi,\psi\>_{L^2(X\times\{r\})}-\<A\psi,\psi\>_{L^2(X\times\{-r\})}\;.
  \end{equation*}
  Let us assume that~$\psi$ satisfied the modified APS-boundary conditions.
  In particular, the two boundary terms are nonnegative.
  The Schr\"odinger-Lichnerowicz formula on~$X$ implies
  that
  \begin{equation*}
    \norm{\nabla^z\psi}^2_{L^2(Z_r\sqcup Z'_r)}=\norm{A\psi}^2_{L^2(Z_r\sqcup Z'_r)}
    +\norm{\nabla^z_{\del_t}\psi}^2_{L^2(Z_r\sqcup Z'_r)}\;.
  \end{equation*}
  Because~$\<\mathcal R\psi,\psi\>_{L^2}=O(e^{-\ell/C})$
  by~\eqref{ModifiedDirac1} and \eqref{ModifiedDirac2},
  we conclude that if~$\abs\lambda$ is sufficiently small,
  then~$\psi|_{Z_r\sqcup Z'_r}$ must have a nonvanishing
  component~$\psi_0\in C^\infty([-r,r],\ker A\oplus\ker A)
  \subset{\Gamma(SZ_r)\oplus\Gamma(SZ'_r)}$.

  First note that because~$\chi(t)^2+\chi(-t)^2=1$
  and~$\trang(t)=\chi'(t)\chi(-t)+\chi(t)\chi'(-t)$,
  we have
  \begin{equation*}
    \frac\del{\del t}\arcsin\bigl(\chi(t)\bigr)
    =\frac{\chi'(t)}{\chi(-t)}
    =\trang(t)\;.
  \end{equation*}
  For a given~$\alpha_0\in\R$, let us abbreviate
  \begin{equation*}
    \alpha(t)=(1-z)\,\arcsin\Bigl(\chi\Bigl(\frac tr\Bigr)\Bigr)+\alpha_0\;.
  \end{equation*}
  By~\eqref{3.22b} and~\eqref{eq:Dzdef},
  a $\ker A\oplus\ker A$-valued $\lambda$-eigenspinor~$\psi_0$ on~$Z_r$
  has to satisfy the linear ordinary differential equation
  \begin{equation*}
    \lambda\psi_0
    =
    \begin{pmatrix}
      \gamma\,\frac\del{\del t}
      &\frac{1-z}r\,\omega\bigl(\frac tr\bigr)\gamma\\
      -\frac{1-z}r\,\omega\bigl(\frac tr\bigr)\gamma
      &\gamma\,\frac\del{\del t}
    \end{pmatrix}
    \psi_0\;.
  \end{equation*}
  The space of solutions is spanned by spinors of the form
  \begin{equation}\label{3.10}
      \begin{pmatrix}
        e^{-\lambda t\gamma}\cos(\alpha(t))\,h_0\\
        e^{-\lambda t\gamma}\sin(\alpha(t))\,h_0
      \end{pmatrix}
      \in C^\infty([-r,r],\ker A\oplus\ker A)
  \end{equation}
  with~$h_0\in\ker A$ and~$\alpha_0\in\R$.
  Prescribing~$\psi_0(-r)=\bigl(\begin{smallmatrix}\psi_1(-r)\\\psi_2(-r)\end{smallmatrix}\bigr)$ at~$t=-r$,
  the value at~$t=r$ is given by
  \begin{equation}\label{3.10a}
    \psi_0(r)=
    \begin{pmatrix}
      \psi_1(r)\\\psi_2(r)
    \end{pmatrix}
    =e^{-2\lambda r\gamma}
    \begin{pmatrix}
      \phantom{-}\sin\frac{z\pi}2&\cos\frac{z\pi}2\\[\smallskipamount]
      -\cos\frac{z\pi}2&\sin\frac{z\pi}2
    \end{pmatrix}
    \begin{pmatrix}
      \psi_1(-r)\\\psi_2(-r)
    \end{pmatrix}\;.
  \end{equation}

  Let us now consider the case~$\lambda=0$.
  By construction of the deformation of~$D$ in Section~\ref{Abs3.2},
  the sections~$s$ and~$c_{v_\pm}s$ are harmonic on~$M_{\pm,\ell,r}\setminus Z_r$
  for all~$z\in[0,1]$.
  Arguing by partial integration as in~\eqref{eq:kerAest1},
  we see that the first component of~$\psi_0(\pm r)$
  must be perpendicular to~$\gamma s|_{X\times\{\pm r\}}$
  and~$\gamma c_{v_\pm}s|_{X\times\{\pm r\}}$.
  Combining this with our choice of boundary conditions on~$\del Z'_r$,
  we have
  \begin{equation}\label{3.11}
    \psi_0(-r)\in \lagr_{D_-}\oplus \lagr_{D_+}^\perp\qquad\text{and}\qquad
    \psi_0(r)\in \lagr_{D_+}\oplus \lagr_{D_-}^\perp\;.
  \end{equation}

  We will show that~$\psi_0(r)=0$ is the only solution to~\eqref{3.10a}
  and~\eqref{3.11}.
  Recall that by~\eqref{eq:cvpms}, we have~$c_{v_+}s=e^{\pm\thet\gamma}c_{v_-}s$.
  By~\eqref{3.11}, there are~$v$, $w$, $x$, $y\in\R$ such that
  \begin{equation*}
    \psi_0(-r)=
    \begin{pmatrix}
      vs+wc_{v_-}s\\x\gamma s+y\gamma c_{v_+}s
    \end{pmatrix}\;.
  \end{equation*}
  By~\eqref{3.10a} and~\eqref{eq:cvpms}, this implies that
  \begin{equation*}
    \psi_0(r)=
    \begin{pmatrix}
      \sin\frac{z\pi}2\,\bigl(vs+we^{\mp\thet\gamma}c_{v_+}s\bigr)
      +\cos\frac{z\pi}2\bigl(x\gamma s+y\gamma c_{v_+}s\bigr)\\[\smallskipamount]
      -\cos\frac{z\pi}2\,\bigl(vs+wc_{v_-}s\bigr)
      +\sin\frac{z\pi}2\,\bigl(x\gamma s+y\gamma e^{\pm\thet\gamma}c_{v_-}s\bigr)
    \end{pmatrix}\;.
  \end{equation*}
  This is an element of~$\lagr_{D_+}\oplus \lagr_{D_-}^\perp$ if and only if the
  following four conditions hold.
  \begin{align*}
    x\cos\frac{z\pi}2&=0\;,&
    \mp w\sin\thet\sin\frac{z\pi}2+y\cos\frac{z\pi}2&=0\;,\\
    -v\cos\frac{z\pi}2&=0\;,&
    -w\cos\frac{z\pi}2\mp y\sin\thet\sin\frac{z\pi}2&=0\;.
  \end{align*}
  Because~$\sin\thet\ne0$ by assumption,
  we conclude that~$w=y=0$ independent of~$\thet$ and~$z$.
  Hence, the only nonzero solutions arise when~$z=1$.
  They correspond to the expected solutions~$s\in\ker D_{M,\ell,r}$
  and~$\gamma s\in\ker D_{Z'_r}$.
  
  For~$z$ close to~$1$ and~$\abs\lambda$ small,
  the $\lambda$-eigenspinors of~$\D_z$
  are still close to~$\ker\D_1$.
  So let~$x+v$ be a $\lambda$-eigenspinor of~$\D_z$ with~$x\in\ker\D_1$
  and~$v\perp\ker\D_1$ small,
  then by~\eqref{eq:Dzdef}, we compute
  \begin{equation*}
    \lambda\norm x^2_{L^2}=\bigl\<\bigl(\D_1-(1-z)G\bigr)(x+v),x\bigr\>_{L^2}
    =-(1-z)\<G(x+v),x\>_{L^2}\;.
  \end{equation*}
  The kernel of~$\D_1$ is spanned by~$\binom s0$ and~$\binom 0{\gamma s}$.
  By~\eqref{eq:Gdef}, using~$\omega\le0$,
  we see that for~$1-z$ sufficiently small,
  one of the eigenvalues is positive and one is negative.
  Hence, for~$z\nearrow 1$,
  one eigenvalue approaches~$0$ from above and one from below.
  We get~\eqref{3.8} and hence Theorem~\ref{thm3.4}, because
  \begin{equation*}
    2\SF\bigl((\D_z)_{z\in[0,1]}\bigr)=-2=\dim\ker \D_0-\dim\ker \D_1\;.\qedhere
  \end{equation*}
\end{proof}

\subsection{The Spectral Flow of the Deformation}
Let~$D_{\ell,r,w}$ be a deformation of the spin Dirac
operator on the manifolds~$M_{\ell,r}$ constructed as in~\eqref{ModifiedDirac2}.
We show that~$D_{\ell,r,w}$ has one-dimensional kernel for all~$w\in[0,1]$
if~$\ell$ and~$r$ are sufficiently large.

\begin{prop}\label{prop3.13}
  There exists a constant~$C$ such that for~$\ell$ and~$r$ sufficiently large,
  $$\eta(D_{\ell,r,0})=\eta(D_{\ell,r,1})+O\bigl(e^{-\ell/C}\bigr)
  \qquad\text{and}\qquad \dim\ker(D_{\ell,r,w})%
  =1\quad\text{for all }w\in[0,1]\;.$$
\end{prop}

\begin{proof}
  We fix~$\ell\gg 1$ sufficiently large and~$r=\ell$ so that the conclusion
  of Proposition~\ref{prop:smallev} holds.
  The variation of the $\eta$-invariant in~$w$
  consists of a local variation term in~$\R$
  and a spectral flow contribution in~$\Z$.
  We can apply Theorem~\ref{Thm2.1} to~$M_{2\ell-1,1}$
  and compare with~$M_{\ell,\ell}$ to see that
  the local variation is of order~$O(e^{-\ell/C})$. %
  The spectral flow contribution vanishes because of the second claim.

  It therefore remains to establish the second claim.
  By the construction of~$D_{\ell,r,w}$ following~\eqref{ModifiedDirac2},
  we have~$s\in\ker D_{\ell,r,w}$ for all~$w\in[0,1]$ and
  all sufficiently large~$\ell$ and~$r$.
  We also know that~$D_{\ell,r,w}-D_{\ell,r,1}$ is an operator of degree~$0$
  and of order~$O(e^{-\ell/C})$.
  Hence, if we have chosen~$\ell\gg1$ large enough and~$r=\ell$,
  the Rayleigh quotient of~$D_{\ell,r,1}$ on~$\ker D_{\ell,r,w}$
  is smaller than~$\frac c{2\ell}$.
  Hence, $\dim\ker D_{\ell,r,w}>1$ would
  contradict Proposition~\ref{prop:smallev}.
\end{proof}

\subsection{The \texorpdfstring{$\eta$}{eta}-Invariant
of the Spin Dirac Operator}\label{Abs3.6}
By Remark~\ref{rem:par},
the space of real harmonic spinors on~$\X$ is the four-dimensional
vector space spanned by~$s$, $c_{\anglen_-}s$, $c_{\anglex_-}s$ and~$\gamma s$.
The Lagrangians~$\lagr_{D_-}$ and~$\lagr_{D_+}$ are given by equation~\eqref{3.23}.

\begin{thm}\label{thm3.7}
  Let~$\rho=\pi-2\thet$. Then
  \begin{equation*}
    m_{\ker A}(\lagr_{D_+},\lagr_{D_-})=\frac\rho\pi\;.
  \end{equation*}
\end{thm}

\begin{proof}
Let~$A_\pm$ denote the reflections along~$\lagr_{D_\pm}$,
then~$-A_+A_-$ acts as~$-\id$ on the subspace spanned by~$s$ and~$\gamma s$,
and rotates the subspace spanned by~$c_{\anglen_-}s$ and~$c_{\anglex_-}s$
by~$\pi-2\thet=\rho$.
Hence, the eigenvalues of~$-A_+A_-$ on the $(-i)$-eigenspace of~$\gamma$
are~$-1$ and~$e^{-\rho i}$, and by~\eqref{Mdef},
the result follows.
\end{proof}

\subsection{A Proof of Theorem~\texorpdfstring{\ref{GluingThm}}{1}}
We can now prove the main result of this paper.

\begin{dfn}\label{Def:nupm}
  Let~$B_{M_{\pm,\ell}}$, $D_{M_{\pm,\ell}}$ denote the odd signature operator
  and the modified Dirac operator on~$M_{\pm,\ell}$, then define
  \begin{equation*}
    \bar\nu(M_\pm,g)
    =\lim_{\ell\to\infty}\bigl(3\eta_{\mathrm{APS}}(B_{M_\pm,\ell};\lagr_{B_\pm})
    -24\eta_{\mathrm{APS}}(D_{M_\pm,\ell};\lagr_{D_\pm})\bigr)\;.
  \end{equation*}
\end{dfn}

\begin{proof}[Proof of Theorem~\ref{GluingThm}]
  We note that the parameter~$r$ used in the gluing formula above does
  not affect the $\eta$-invariants~$\eta_{\APS}(D_{M_\pm,\ell,r};\lagr_{D_\pm})$
  by the variation formula for $\eta$-invariants on manifolds with boundary
  in~\cite{BCh4} and~\cite{daifreed}.
  This is relevant if~$k_+>2$ or~$k_->2$.
  By Definition~\ref{MainDef} of~$\bar\nu(M,g)$,
  our result follows by
  combining Theorems~\ref{thm2.6}, \ref{thm3.4}, \ref{thm3.7}
  and Proposition~\ref{prop3.13}.
\end{proof}

\begin{rmk}\label{BarNuBoundedRem}
  By construction, $\rho\in(-\pi,\pi)$.
  The configuration angles~$\alpha^-_1$, \dots, $\alpha^-_{19}$
  in Definition~\ref{def:angles} are either~$0$, $\pi$,
  or occur in pairs of angles with opposite sign.
  Of each such pair, at most one angle belongs to~$(\pi-\abs\rho,\pi)$,
  see the proof of Proposition~\ref{Mprop} for details.
  In particular,
	$$0\le\#\bigl\{\,j
		\bigm|\alpha_j^-\in\{\pi-\abs\rho,\pi\}\,\bigr\}
	+2\#\bigl\{\,j
	\bigm|\alpha_j^-\in(\pi-\abs\rho,\pi)\,\bigr\}\le 19\;.$$
  Hence by Definition~\ref{def2.6}, $m_\rho(\latc;N_+,N_-)\in(-3,54)$
  if~$\rho>0$, and~$m_\rho(\latc;N_+,N_-)\in(-54,3)$ if~$\rho<0$.
  If~$\rho=0$, then~$m_\rho(\latc;N_+,N_-)=0$.
  Because~$\rho\in(-\pi,\pi)$, we get the estimate
	$$-75<-72\frac\rho\pi+3m_\rho(\latc;N_+,N_-)<75\;.$$
  In particular,
  we cannot use the extra-twisted connected sum construction
  with~$k_\pm\le 2$ to produce families of $G_2$-manifolds
  with infinitely many different values of~$\bar\nu$.
\end{rmk}

\end{document}